\documentclass[12pt,reqno]{amsart}

\usepackage{amsthm,amsmath,amssymb,amsfonts}
\usepackage{epsfig}
\usepackage{url,hyperref}
\usepackage{tikz}
\usetikzlibrary{calc}
\usetikzlibrary{patterns}

\makeatletter \@addtoreset{equation}{section}

\makeatother

\renewcommand{\subsection}[1]{\refstepcounter{subsection}
\par\medskip\par\noindent\textbf{\thesubsection.~#1}}

\makeatletter
\def\Maketitle{{\def\newpage{}\maketitle}}
\def\Appendix{\appendix
  \def\@seccntbformat##1{Appendix~\csname the##1\endcsname.~~}}
\makeatother

\newtheorem{Proposition}{Proposition}[section]
\newtheorem{Lemma}{Lemma}[section]
\newtheorem*{Lemma*}{Lemma}

\newtheorem{Theorem}{Theorem}
\newtheorem*{Theorem*}{Theorem}

\theoremstyle{definition}

\newtheorem{Remark}{Remark}[section]
\newtheorem{Example}{Example}[section]

\newcommand{\rto}{\!\to\!}
\def\qbinom#1#2#3{\genfrac[]{0pt}{}{#1}{#2}_{#3}}
\newcommand{\calF}{\mathcal{F}}

\begin{document}


\begin{title}[Plane partitions with a ``pit'']
{Plane partitions with a ``pit'': generating functions and representation theory}
\end{title}

\author{M. Bershtein, B. Feigin, G. Merzon}

\address{Landau Institute for Theoretical Physics, Chernogolovka, Russia
\newline 
Skolkovo Institute of Science and Technology, Moscow, Russia,
\newline
National Research University Higher School of Economics, Moscow, Russia,
\newline 
Institute for Information Transmission Problems,  Moscow, Russia
\newline 
Independent University of Moscow, Moscow, Russia}
\email{mbersht@gmail.com}

\address{National Research University Higher School of Economics, Moscow, Russia}
\email{borfeigin@gmail.com}

\address{Moscow Center for Continuous Mathematical Education, Moscow, Russia}
\email{merzon@mccme.ru}

\begin{abstract}
We study plane partitions satisfying condition $a_{n+1,m+1}=0$ (this condition is called ``pit'') and asymptotic conditions along three coordinate axes. We find the formulas for generating function of such plane partitions.

Such plane partitions label the basis vectors in certain representations of quantum toroidal $\mathfrak{gl}_1$ algebra, therefore our formulas can be 	interpreted as the characters of these representations. The resulting formulas resemble formulas for characters of tensor representations of Lie superalgebra $\mathfrak{gl}_{m|n}$. We discuss representation theoretic interpretation of our formulas using $q$-deformed $W$-algebra $\mathfrak{gl}_{m|n}$.
\end{abstract}

\dedicatory{To Sasha Beilinson on the occasion of his birthday. At least one of the authors owes him a lot.}


\Maketitle

\section{Introduction}

In this paper we study certain problems of enumerative combinatorics of 3d Young diagrams, which are motivated by  representation theory. 

It is convenient to identify 3d Young diagrams with plane partitions, i.e., collection of nonnegative integers  $a_{i,j}$ such that $a_{i,j}\geq a_{i+1,j}$, $a_{i,j}\geq a_{i,j+1}$ and all but a finite number of $a_{i,j}$ equals 0. Later we will also consider more general plane partitions.

Denote by $|a|=\sum a_{i,j}$, i.e., the number of boxes in the corresponding 3d Young diagram. For any set $\mathcal{A}$ of plane partitions, define its generating function by $\sum_{a \in \mathcal{A}} q^{|a|}$. Such functions were extensively studied in enumerative combinatorics, for example, one of MacMahon's formulas has the form (see e.g. \cite[1.5 ex. 13(d)]{Macdonald Book})
\[
\sum_{\{a|a_{n+1, 1}=0\}}q^{|a|}=q^{-\binom{n}{3}} \frac{V(1,q,\ldots,q^{n-1})}{(q)_\infty^n},
\]
where $V(x_1,\ldots,x_n)=\prod_{i<j}(x_i-x_j)$ is the Vandermonde product and $(q)_\infty=\prod_{k=1}^\infty (1-q^k)$. The limit $n \rightarrow \infty$ gives well-known MacMahon formula for the generating function of all plane partitions: $\sum_{a}q^{|a|}=\prod_{k=1}^\infty (1-q^k)^{-k}$.

We study plane partitions satisfying the condition 
\begin{align}
a_{n+1,m+1}=0.  \label{eq:pit}
\end{align}
We will call such condition ``pit'' in box $(n+1,m+1)$. Moreover we will consider plane partitions $a=\{a_{ij}\}$ with an infinite number of non-zero $a_{i,j}$ and some of $a_{i,j}$ equal to $\infty$, satisfying the following asymptotic conditions
\begin{align}
\textbf{1.}\; \lim_{j\rightarrow \infty} a_{i,j} = \nu_i, \qquad &\textbf{2.}\; \lim_{i\rightarrow \infty} a_{i,j} = \mu_j, \qquad \textbf{3.}\; a_{i,j} = \infty\; \text{iff $(i,j) \in \lambda$,} \label{eq:asymp} 
\end{align}
where $\nu,\mu, \lambda$ are partitions (see Fig. \ref{fig:planepart}).

\begin{figure}[h]
\begin{tikzpicture}[scale=.32,nodes={font=\small},baseline={([yshift=-.5ex]current bounding box.center)}]
\foreach \m [count=\y] in {%
         {15,15,15,15,4,4,3,3,3,3,3,3,3},{15,15,15,5,4,2,1,1,1,1,1,1,1},
         {15,15,5,4,4,1},{5,4,4,2,1},{4,2,1},
         {2,2,1},{2,2,1},{2,2,1},{2,2,1},{2,2,1},{2,2,1},{2,2,1},{2,2,1}%
  } {
  \foreach \n [count=\x] in \m {
  \ifnum \n>0
      \foreach \z in {1,...,\n}{
        \draw[fill=gray!30] (\x+1,\y,\z)--(\x+1,\y+1,\z)--(\x+1, \y+1, \z-1)--(\x+1, \y, \z-1)--cycle;
        \draw[fill=gray!40] (\x,\y+1,\z)--(\x+1,\y+1,\z)--(\x+1, \y+1, \z-1)--(\x, \y+1, \z-1)--cycle;
        \draw[fill=gray!10] (\x,\y,\z)  --(\x+1,\y,\z)  --(\x+1, \y+1, \z)  --(\x, \y+1, \z)--cycle;  
      }
 \fi
 }
}
\draw[thick] (5,1,4)--(5,1,5)--(6,1,5) node[below right=-1mm]{$m=n=4$}--(6,1,4)--cycle (5,1,4)--(6,1,5) (5,1,5)--(6,1,4);
\draw[thick] (14,1,0)--(14,3,0)--(14,3,1)--(14,2,1)--(14,2,3)--(14,1,3)--cycle;
\draw[thick] (1,14,0)--(4,14,0)--(4,14,1)--(3,14,1)--(3,14,2)--(1,14,2)--cycle;
\draw[thick] (1,1,15)--(5,1,15)--(5,2,15)--(4,2,15)--(4,3,15)--(3,3,15)--(3,4,15)--(1,4,15)--cycle;
\draw[->] (14,1,0) node[below=5mm,right=-2mm] {$\nu=(2,1,1)$}--(15.5,1,0) node[right] {\small 2};
\draw[->] (1,14,0) node[above=4mm,right=1mm] {$\lambda=(3,2)$}--(1,15.5,0) node[above] {\small 3};
\draw[->] (1,1,15) node[below=4mm,right=-2mm] {$\mu=(3,3,2,1)$}--(1,1,18) node[below left] {\small 1};
\end{tikzpicture}%
\quad%
\begin{tikzpicture}[scale=.6,nodes={font=\small},baseline={([yshift=-.5ex]current bounding box.center)}]
\draw[thick] (4,12)--(5,12)--(5,11)--(4,11)--cycle (4,12)--(5,11) (4,11)--(5,12);
\draw[thick] (0,4)--(0,14)--(2,14)--(2,15)--(3,15)--(3,16)--(11,16) (11,12)--(4,12)--(4,4);
\clip (0,4)--(0,16)--(11,16)--(11,12)--(4,12)--(4,4)--cycle;
\draw[help lines] (0,0) grid (14,16);
\foreach \m [count=\y] in {%
    {\infty,\infty,\infty,4,4,3,2,2,2,\dots,2},
    {\infty,\infty,4,4,3,2,1,1,1,\dots,1},
    {5,4,4,3,3,1,1,1,1,\dots,1},
    {5,4,4,3,3,1,0,0,0,\dots,0},
    {4,3,3,2},
    {3,3,2,1}, {3,3,2,1}, {3,3,2,1}, {3,3,2,1}, {3,3,2,1},
    {\dots,\dots,\dots,\dots}, {3,3,2,1}%
  } {
  \foreach \n [count=\x] in \m {
     \node at (\x-.5,16.5-\y) {$\n$};
  }
}
\end{tikzpicture}
\caption{\label{fig:planepart}}
\end{figure}

We denote by $\chi^{n,m}_{\mu,\nu,\lambda}(q)$ the generating function of plane partitions which satisfy \eqref{eq:pit}, \eqref{eq:asymp} (for the definition of $|a|$ see \eqref{eq:pp:grad}). It follows from these conditions that $l(\nu)\leq n$, $l(\mu)\leq m$, and $\lambda_{n+1} < m+1$.

Note that asymptotic conditions \eqref{eq:asymp} appear in the theory of topological vertex \cite{Okounkov Vafa:2003}. The condition \eqref{eq:pit} also appeared in the strings theory, see \cite{Jafferis}. Our motivation comes from representation theory, which we discuss below.

In some particular cases the formulas for functions $\chi^{n,m}_{\mu,\nu,\lambda}(q)$ were known before. In order to write down the answers we need some notation. By $\rho_n$ we denote the partition $(n-1,n-2,\dots,1,0)$, we omit the index $n$ and write just $\rho$ if the number of parts is clear from the context.
For any partition of no greater then $n$ parts $\lambda=(\lambda_1,\lambda_2,\dots,\lambda_n)$ by $q^{\lambda+\rho}$ we denote $(q^{\lambda_1+n-1},\ldots, q^{\lambda_n})$. By $a_{\lambda+\rho}(x_1,\dots,x_n)$ we denote the antisymmetric polynomial:
\begin{equation}
a_{\lambda+\rho}(x_1,\dots,x_n)=\det\begin{pmatrix}
x_i^{\lambda_j+n-j}
\end{pmatrix}_{i,j=1}^n. \label{eq:a_lambda}
\end{equation}

The formula for $\chi^{n,m}_{\mu,\nu,\lambda}(q)$ is known in the case $m=0$, i.e., when the ``pit'' is located near the ``wall''. Namely
\begin{equation} \label{eq:Wn:char}
\chi_{\varnothing,\nu,\lambda}^{n,0}(q) =
\frac{q^{\sum_{i=1}^n (\lambda_i+n-i)(\nu_{i}+n-i) }}{(q)_\infty^n} a_{\nu+\rho}(q^{-\lambda-\rho}),
\end{equation}
see, for example, \cite[Theorem 4.6]{Feigin_Miwa:2010}. Clearly this is a generalization of the MacMahon formula given above. Another known case is where two asymptotic conditions vanish $\lambda=\mu=\varnothing$, see \cite{Feigin Mutafyan1}, \cite{Feigin Mutafyan2}.

In our paper we find the formula for $\chi^{n,m}_{\mu,\nu,\lambda}(q)$ in general case. Actually we prove three formulas, which are algebraically equivalent but have different form and meaning. They are given in Theorems \ref{th:1}, \ref{th:2}, \ref{th:2'} below. Here we give the simplest (but already new) particular case of Theorem \ref{th:2}
\begin{equation}\label{eq:chi:m=n}
\chi_{\mu,\nu,\varnothing}^{n,n}(q)= \!\!\!\!\sum_{A_1>A_2>\ldots>A_n\geq 0} \!\!\!\!\!\! (-1)^{\sum\limits_{i=1}^n A_i}q^{\sum\limits_{i=1}^n \binom{A_i+1}2} \frac{a_{\nu+\rho}(q^A)a_{\mu+\rho}(q^{-A})}{(q)_\infty^{2n}}, 
\end{equation}
Note that each summand is a product of two expressions on the right side of \eqref{eq:Wn:char} (up to the factor of the form $(-1)^{\ldots}q^{\ldots}$)

Since our three formulas are algebraically equivalent it is enough to prove any of them. We give two different combinatorial proofs, one for Theorem \ref{th:1} and one for Theorem \ref{th:2'}. These proofs are simpler than ones of particular cases given in \cite{Feigin Mutafyan1}, \cite{Feigin Mutafyan2}.  

The first proof is based on a bijection between plane partitions and collections of non crossing paths. The number of such collections is computed using Lindstr\"{o}m--Gessel--Viennot lemma \cite{Lind},\cite{Gessel Viennot}. Such proof gives a determinantal expression for 
$\chi^{n,m}_{\mu,\nu,\lambda}(q)$, see Theorem~\ref{th:1}.

In the second proof we interpret conditions \eqref{eq:pit},\eqref{eq:asymp} as a definition of certain infinite dimensional polyhedron. We compute the generating function of integer points in this polyhedron as a sum of contribution of vertices, using Brion theorem \cite{Brion}. Such proof gives a ``bosonic formula''\footnote{Usually a formula is called bosonic if it equals a linear combination of characters of algebra of polynomials. In our case bosonic formula is a combination of terms $q^\Delta/(q)_\infty^{n+m} $.} for 
$\chi^{n,m}_{\mu,\nu,\lambda}(q)$, see Theorem \ref{th:2'}.

The conditions \eqref{eq:pit},\eqref{eq:asymp} appeared in \cite{Feigin_Miwa:2011} in the context of representation theory of quantum toroidal algebra $U_{\vec{q}}(\ddot{\mathfrak{gl}}_1)$. Namely, plane partitions which satisfy these conditions label a basis in MacMahon modules $\mathcal{N}^{n,m}_{\mu,\nu,\lambda}(v)$ over $U_{\vec{q}}(\ddot{\mathfrak{gl}}_1)$. Therefore $\chi^{n,m}_{\mu,\nu,\lambda}(q)$ is the character of the representation $\mathcal{N}^{n,m}_{\mu,\nu,\lambda}(v)$. It is natural to ask for representation theoretic interpretation of our character formulas. 

We conjecture that there exist resolutions of $\mathcal{N}^{n,m}_{\mu,\nu,\lambda}(v)$ such that their Euler characteristics  coincide with our character formulas. In such cases we say that the resolution \textit{is a materialization} of the character formula. For example, the BGG resolution \cite{BGG} is a materialization of the Weyl character formula. Zelevinsky constructed complex which is a materialization of the Jacobi--Trudi formula for Schur polynomials \cite{Zel}.

Our formulas for functions  $\chi^{n,m}_{\mu,\nu,\lambda}(q)$ resemble the formulas for characters of tensor representations of Lie superalgebra $\mathfrak{gl}_{m|n}$. This similarity can be explained by the fact that the representations $\mathcal{N}^{n,m}_{\mu,\nu,\lambda}(v)$ are actually representations of certain $q$-deformed $W$-algebra, which we call $\mathrm{W}_{\vec{q}}(\dot{\mathfrak{gl}}_{n|m})$. 

Such $W$-algebras appear as follows. There is an easy (but not written in the literature) fact that $\mathcal{N}^{n,m}_{\mu,\nu,\lambda}(v)$ is a subquotient of Fock representation of 
$U_{\vec{q}}(\ddot{\mathfrak{gl}}_1)$. This Fock module is a tensor product of $n+m$ basic Fock modules, see the formula \eqref{eq:Fnm}. On such Fock modules the algebra $U_{\vec{q}}(\ddot{\mathfrak{gl}}_1)$ acts through the its quotient which is called $W$-algebra. This $W$-algebra commutes with the certain intertwining operators, which are called screening operators. The structure of the screening operators in our case suggests the name $\mathrm{W}_{\vec{q}}(\dot{\mathfrak{gl}}_{n|m})$. 

There exists the conformal limit $\vec{q}\rightarrow (1,1,1)$ of screening operators and we denote the limit of the algebra $W$ by $\mathrm{W}(\dot{\mathfrak{gl}}_{n|m})$. For $m=0$ this algebra coincides with the algebra $\mathrm{W}(\dot{\mathfrak{gl}}_{n})$ \cite{Fateev Lukyanov}. The algebras $\mathrm{W}(\dot{\mathfrak{gl}}_{n|1})$ conjecturally coincide with the ``zero momentum'' sector of $W_n^{(2)}$ algebras introduced in \cite{Feigin Semikhatov}. We did not find reference for generic $n,m$.\footnote{\textbf{Note added}: After the first version of the paper appeared in the arXiv these   $W$ algebras (as well as more generic, see Remark \ref{rem:Wmnk}) were studied in the framework of conformal field theory \cite{Litvinov} and supersymmetric gauge theory \cite{Gaiotto}.
} Note that our $W$-algebras differ from ones introduced in \cite{KRW}.

Standard statement in the theory of vertex algebras is an equivalence of the abelian categories of certain representations of vertex algebra and certain representations of quantum group. This is a statement similar to Drinfeld--Kohno or Kazhdan--Lusztig theorems. We conjecture that under this equivalence $\mathrm{W}(\dot{\mathfrak{gl}}_{n|m})$ is related to the product of quantum groups $U_q\mathfrak{gl}_{n|m}\otimes U_{q'}\mathfrak{gl}_n\otimes U_{q''}\mathfrak{gl}_m$ for certain $q,q',q''$. And representations $\mathcal{N}^{n,m}_{\mu,\nu,\lambda}(v)$ under this equivalence go to the tensor products $L^{(n|m)}_\lambda \otimes L^{(n)}_\nu\otimes L^{(m)}_\mu $, where  $L^{(n)}_\nu$ and $L^{(m)}_\mu$ are finite dimensional irreducible representations of $U_{q'}\mathfrak{gl}_n$ and $U_{q''}\mathfrak{gl}_m$, correspondingly, and $L^{(n|m)}_\lambda$ is a tensor representation of $U_q\mathfrak{gl}_{n|m}$ (representations $L^{(n|m)}_\lambda, L^{(n)}_\nu, L^{(m)}_\mu$ also depend on $v$ but we omit this dependence for simplicity).

\par\medskip\par\noindent\textbf{Plan of the paper.}
In Section \ref{sec:MainResults} we give precise statements of our main results for $\chi^{n,m}_{\mu,\nu,\lambda}(q)$ with necessary notation and comments. The remaining sections \ref{sec:Paths}, \ref{sec:Brion}, \ref{sec:alg} are independent of each other. Sections \ref{sec:Paths} and \ref{sec:Brion} are devoted to the combinatorial proofs based on Lindstr\"{o}m--Gessel--Viennot lemma and Brion theorem, correspondingly. Section~\ref{sec:alg} is devoted to the algebraic discussion. First we give a definition of the appropriate $W$-algebras in terms of screening operators. Conjectural materializations of our character formulas are discussed in subsection~\ref{ssec:Fnm}. Relation to quantum group $U_q\mathfrak{gl}_{n|m}\otimes U_{q'}\mathfrak{gl}_n\otimes U_{q''}\mathfrak{gl}_m$ is discussed in subsection~\ref{ssec:glnm}.

\par\medskip\par\noindent\textbf{Acknowledgments.} We thank   A. Babichenko, E.~Gorsky, A.~Kirillov, A.~Litvinov, I.~Makhlin, G.~Mutafyan, G.~Olshanski, Y.~Pugai,  A.~Sergeev for interest to our work and discussions. We are grateful to anonymous referee for the careful reading and huge number of valuable comments.

This work has been funded by the Russian Academic Excellence Project ‘5-100’. M.~B.~acknowledges the financial support of Simons-IUM fellowship, RFBR grant mol\_a\_ved 15-32-20974 and Young Russian Math Contest.

\section{Main Results} \label{sec:MainResults}

\subsection{\!\!} \label{ssec:Notation}
Due to asymptotic conditions \eqref{eq:asymp} $a_{i,j}\geq \nu_i$ and $a_{i,j}\geq \mu_j$. In order to define the grading $|a|\sim \sum a_{i,j}$, we need to subtract these asymptotic values $\nu_i$, $\mu_j$. 

We will use the following definition
\begin{equation}\label{eq:pp:grad}
|a|=\sum_{i-n\leq j-m,\; (i,j) \not\in \lambda} (a_{i,j}-\nu_i)+ \sum_{i-n> j-m,\; (i,j) \not\in \lambda} (a_{i,j}-\mu_j).
\end{equation}
Note that this definition of grading is not invariant under  $m, \mu$ $\leftrightarrow$ $n, \nu$ symmetry. Geometrically the definition \eqref{eq:pp:grad} can be restated as follows. We draw a staircase line from the point $(n,m)$ as on the picture below. This line divides the base of the plane partition $a$ into two parts. We subtract $\nu_i$ from cells in the upper part and $\mu_j$ from cells in the left part, see Fig. \ref{fig:grading}.

\begin{figure}[h]
\begin{tikzpicture}[scale=.75,nodes={font=\small}]
\draw[help lines] (0,-.5) grid (7.5,9);
\fill[black!25!white] (0,4)--(1,4)--(1,6)--(2,6)--(2,8)--(3,8)--(3,9)--(0,9)--cycle;
\fill[black!25!white] (4,-.5)--(4,4)--(7.5,4)--(7.5,-.5)--cycle;
\draw[ultra thick] (0,-.5)--(0,4)--(1,4)--(1,6)--(2,6)--(2,8)--(3,8)--(3,9)--(7.5,9);
\draw[ultra thick] (4,-.5)--(4,4)--(7.5,4);
\draw[very thick,densely dashed] (4,4)--(3,4)--(3,5)--(2,5)--(2,6);
\foreach \y/\minx in {5/3, 4/2, 3/2, 2/2, 1/3} {
   \foreach \x in {\minx,...,6} {
      \node at (\x+.5,9-\y+.5) {$-\nu_\y$};
   }
} 
\foreach \x/\maxy in {1/4, 2/6, 3/5, 4/4} {
   \foreach \y in {\maxy,...,1} {
      \node at (\x-.5,\y-.5) {$-\mu_\x$};
   }
} 
\end{tikzpicture}
\caption{\label{fig:grading}}
\end{figure}

Let $r=\min\{t|\lambda_{n-t}\geq m-t\}$, $0 \leq r \leq \min\{n,m\}$. Geometrically $r$ is the number of horizontal steps in the staircase line starting from the point $(m,n)$. In the Fig. \ref{fig:grading} we have $r=2$. Note that this number $r$ has an interpretation in terms of representation theory of $\mathfrak{gl}(m|n)$, namely, $r$ is called the degree of atypicality (or simply atypicality) of the tensor representation of $\mathfrak{gl}(m|n)$ corresponding to $\lambda$ (see, for example \cite[page 9]{MvdJ}). 

In order to write down the formula for the generating function we parametrize $\lambda$ by some analogue of Frobenius coordinates. We introduce two partitions $\pi, \kappa$ by $\pi_i=\lambda_i-(m-r)$ for $i=1,\ldots,n-r$ and $\kappa_j=\lambda'_j-(n-r)$ for $j=1,\ldots,m-r$, where $\lambda'$ denotes the transpose of the partition $\lambda$.  We denote components of partitions $\nu, \mu, \pi, \kappa$ shifted by $\rho$ by the corresponding capital Latin letters: $N_i=\nu_i+n-i$, $M_j=\mu_j+m-j$, $P_i=\pi_i+(n-r)-i$, $Q_j=\kappa_j+(m-r)-j$.

\subsection{\!\!} In the simplest case $n=1, m=0$ for any asymptotic conditions $\lambda, \nu$ the generating function of partitions $\chi_{\varnothing,\nu,\lambda}^{1,0}(q)$ equals $1/(q)_\infty$ (and similarly for $n=0, m=1$ case). 

Now consider the  $n=m=1$ case. If $\lambda \neq \varnothing$ then the plane partitions decompose into two partitions so the generating function equals $1/(q)_\infty^2$. In the case $\lambda=\varnothing$, there is a clear bijection between plane partitions and $V$-partitions \cite{Stanley} i.e. the $\mathbb{N}$-arrays of integer numbers:
\begin{equation*}
\left(a_0\;\; \begin{matrix}
a_1 & a_2 & a_2 & \ldots\\
b_1 & b_2 & b_2 & \ldots
\end{matrix} \right), 
\end{equation*}
such that $a_0\geq a_1 \geq a_2 \geq \ldots$, $a_0\geq b_1 \geq b_2 \geq \ldots$, $\lim_{i\rightarrow \infty} a_i=\nu_1$, $\lim_{i\rightarrow \infty} b_i=\mu_1$. The weight of the $V$-partition is defined as 
\begin{equation*}
N=\sum_{i \geq 0} (a_i-\nu_1)+\sum_{i \geq 1} (b_i-\mu_1). 
\end{equation*}

\begin{Lemma} \label{lem:V(a,b)}
The generating function of $V$-partitions with asymptotic conditions $\lim_{i\rightarrow \infty} a_i=\nu_1$, $\lim_{i\rightarrow \infty} b_i=\mu_1$ equals
\begin{equation}\label{eq:R:defin}
R(d;q):= \sum_{i=0}^\infty (-1)^{i} \frac{q^{\frac{i(i+1)}{2}}q^{di}}{(q)_\infty^2},
\end{equation}
where $d=\nu_1-\mu_1$.
\end{Lemma}

This lemma was proven in the case $\nu_1=\mu_1$ in \cite[Sec. 2.5]{Stanley} by a kind of inclusion-exclusion argument. Actually this proof works for any values of $\mu_1,\nu_1$. See also \cite[Cor. 5.6]{Feigin_Miwa:2011} for another proof.

Now we can write down the first formula for $\chi_{\mu,\nu,\lambda}^{n,m}(q)$.
\begin{Theorem}
\label{th:1} The generating function $\chi_{\mu,\nu,\lambda}^{n,m}(q)$ is equal to the determinant of a  block matrix of the size $(m+n-r) \times (m+n-r)$
\begin{equation}\label{eq:chi:LGV} 
\chi_{\mu,\nu,\lambda}^{n,m}(q)=\frac{(-1)^{mn-r}q^{\Delta^{n,m}_{\mu,\nu,\lambda}}}{(q)_\infty^{m+n}}\det\begin{pmatrix}
\left(\sum_{a\geq 0}(-1)^a q^{\binom{a+1}2}q^{(N_j-M_i)a}\right)_{\substack{1 \leq i \leq m\\ 1 \leq j \leq n}} & 
\left(q^{-M_iQ_{j}}\right)_{\substack{1 \leq i \leq m\\ 1 \leq j \leq m-r}}\\ 
\left(q^{-N_j(P_{i}+1)}\right)_{\substack{1 \leq i \leq n-r\\ 1 \leq j \leq n}}  & 0
\end{pmatrix},
\end{equation}
where $\Delta^{n,m}_{\mu,\nu,\lambda}=\sum_{j=1}^{m-r}{M_jQ_j}+\sum_{i=1}^{n-r} N_i(P_i+1)$.
\end{Theorem}

Clearly this formula generalizes previous considerations in the  cases $(n,m)= (1,0)$, $(n,m)=(1,1)$ (or $(n,m)=(0,1)$), where determinant becomes $1\times 1$. 

\begin{Remark}
One can think that the formula \eqref{eq:chi:LGV} is similar to the Jacobi--Trudi formula, which expresses generic Schur polynomial $s_\lambda$ in terms of Schur polynomials corresponding to rows (or columns). Since determinant formula \eqref{eq:chi:LGV} also has terms corresponding to thin hooks it is better to think that this formula is similar to Lascoux--Pragacz formula \cite{LascPrag} for Schur polynomials, which is a common generalization of Jacobi--Trudi and Giambelli formulas. We do not know any interpretation of this analogy.
\end{Remark}

The Theorem \ref{th:1} is proven in Section \ref{sec:Paths}. It is natural that the determinant expression for the generating function can be proven using non-intersecting paths and the Lindstr\"{o}m--Gessel--Viennot lemma. 

Let us mention two more special cases where we have only one block in the matrix. In the case of $m=0$, we have $r=0$, $\pi=\lambda$ and after a multiplication on $(q)_\infty^n$ the determinant becomes equal to $a_{\nu+\rho}(q^{-\lambda-\rho})$. So we get the known formula \eqref{eq:Wn:char}.

In the case $m=n$ and $\lambda=\varnothing$, the formula 
\eqref{eq:chi:LGV} simplifies to $\det\Bigl(R(N_j-M_i;q)\Bigr)$. This formula was proven in \cite{Feigin Mutafyan2} (following \cite{Feigin_Miwa:2011}) under the additional assumption that $\mu=\varnothing$.


\subsection{\!\!} The determinant in formula \eqref{eq:chi:LGV} can be calculated.

\begin{Theorem}\label{th:2} The generating function $\chi_{\mu,\nu,\lambda}^{n,m}(q)$ is equal to the sum over $r$-tuples of integer numbers $A_1>A_2>\ldots>A_r\geq 0$
\begin{equation}\label{eq:chi:bos}
\chi_{\mu,\nu,\lambda}^{n,m}(q)=(-1)^{r(m+n)}q^{\Delta^{n,m}_{\mu,\nu,\lambda}} \!\!\!\!\!\!\!\!\sum_{A_1>A_2>\ldots>A_r\geq 0} \!\!\!\!\!\!\!\!\! (-1)^{\sum\limits_{i=1}^r A_i}q^{\sum\limits_{i=1}^r \binom{A_i+1}2} \frac{a_{N}(q^A,q^{-P-1})a_{M}(q^{-A},q^{-Q})}{(q)_\infty^{m+n}}, 
\end{equation}
where $a_N, a_M$ were defined in formula~\eqref{eq:a_lambda} and $\Delta^{n,m}_{\mu,\nu,\lambda}=\sum_{j=1}^{m-r}{M_jQ_j}+\sum_{i=1}^{n-r} N_i(P_i+1)$.

\end{Theorem}

\begin{Lemma}\label{lem:LGV=bos}
	The r.h.s. of \eqref{eq:chi:LGV} and \eqref{eq:chi:bos} are equal.
\end{Lemma}
Our proof of this lemma is based on a direct calculation, which we present in Subsection~\ref{ssec:LGV:lem}.

There are two special cases in which the right side takes a simpler form. These two special cases of the theorem were known. 

First, if $m=0$ then the formula \eqref{eq:chi:bos} reduces to \eqref{eq:Wn:char}. More generally if $r=0$, then base of the plane partition decomposes into two connected components and the formula \eqref{eq:chi:bos} becomes a product $q^{\Delta^{n,m}_{\mu,\nu,\lambda}}{a_{N}(q^{-P-1})a_{M}(q^{-Q})}/{(q)_\infty^{m+n}}$.

In the second case we take $\lambda=\mu=\nu=\varnothing$. Then the functions $a_N$ and $a_M$ reduce to Vandermonde products and we can write (we assume that $n \geq m$)
\begin{multline}
\chi_{\varnothing,\varnothing,\varnothing}^{n,m}(q) \\ =\frac{q^{\Delta^{n,m}_{\varnothing,\varnothing,\varnothing}}}{(q)_\infty^{m+n}}\!\!\!\!\sum_{A_1>A_2>\ldots>A_m\geq 0} \!\!\!\!\!\! (-1)^{\sum\limits_{i=1}^m (A_i-i+1)}q^{\sum\limits_{i=1}^m \binom{A_i+1}2}   {V(q^{m-n},\ldots,q^{-1},q^{A_m},\ldots,q^{A_1})V(q^{-A_1},\ldots,q^{-A_m})} \\=
\frac1{{(q)_\infty^{m+n}}}\sum_{\alpha_1\geq \alpha_2\geq \ldots\geq\alpha_m\geq 0} \!\!\!\!\!\! (-1)^{\sum\limits_{i=1}^m \alpha_i}q^{\sum\limits_{i=1}^m \frac12\alpha_i(\alpha_i+(2i-1))}   \!\!\!\prod_{1\leq i < j \leq m}(1-q^{\alpha_i-\alpha_j-i+j})\!\!\!\prod_{1\leq i < j \leq n}(1-q^{\alpha_i-\alpha_j-i+j}).
\end{multline}
Here we set $\alpha_j=0$ for $j>m$.  This formula coincides with the  \cite[Conjecture 5.10]{Feigin_Miwa:2011} proved in \cite[Theorem 1.2]{Feigin Mutafyan1} by a completely different method.

\subsection{\!\!} The sum in formula \eqref{eq:chi:bos} contains zero terms if one of $A_i$ equals to one of $Q_j$. We want to exclude such terms. The following lemma is standard.
\begin{Lemma} \label{lem:Z=l+l}
For any partition $\lambda$ we have
\[\mathbb{Z}=\{\lambda_j'-j-(n-m)|j\in \mathbb{N}\}\bigsqcup\{i-\lambda_i-(n-m)-1| i \in \mathbb{N}\}.\]
\end{Lemma}
\begin{proof}[Sketch of the proof]
The proof is based on the following construction. Rotate Young diagram corresponding to $\lambda$ by $135^\circ$ and take a projection on $OX$. On the Fig \ref{fig:whiteblack} we give an example for $\lambda=(4,4,4,3,3,1)$.
\begin{figure}
\begin{tikzpicture}[scale=.5,baseline={([yshift=-.5ex]current bounding box.center)}]
\draw[thick,->] (-7.5,0)--(9.5,0) node[right]{};
\foreach \x/\y in {0/1,-1/2,1/2,-2/3,0/3,2/3,-3/4,-1/4,1/4,3/4,
                  -2/5,0/5,2/5,4/5,-1/6,1/6,3/6,2/7} {
     \draw[very thin] (\x,\y-1)--(\x-1,\y)--(\x,\y+1)--(\x+1,\y)--cycle;
}
\draw (-6,6)--(0,0)--(8,8);
\foreach \x/\y in {
                    -3.5/4.5,-2.5/5.5,-1.5/6.5,0.5/6.5,1.5/7.5,4.5/6.5,
                    6.5/6.5,7.5/7.5,8.5/8.5 
                  } {
     \draw[thick] (\x-.5,\y-.5)--(\x+.5,\y+.5);
     \draw[thick,fill=white] (\x,\y) circle (1mm);
     \draw[dotted](\x-.5,\y-.5)--(\x-.5,0) (\x+.5,\y+.5)--(\x+.5,0);
     \draw[thick,fill=white] (\x,0) circle (1mm);
}
\foreach \x/\y in {
                    -0.5/6.5,2.5/7.5,3.5/6.5,5.5/6.5,
                    -4.5/4.5,-5.5/5.5,-6.5/6.5 
                  } {
     \draw[thick] (\x-.5,\y+.5)--(\x+.5,\y-.5);
     \draw[thick,fill=black] (\x,\y) circle (1mm);
     \draw[dotted](\x-.5,\y+.5)--(\x-.5,0) (\x+.5,\y-.5)--(\x+.5,0);
     \draw[thick,fill=black] (\x,0) circle (1mm);
}
\foreach \i in {-7,...,9} {
     \draw[] (\i,.2)--(\i,-.2);
}
\foreach \i in {-1,...,1} {
     \draw[] (\i,.2)--(\i,-.2) node[below] {\small \i};
}
\end{tikzpicture}
\caption{\label{fig:whiteblack}}
\end{figure}

It is easy to see that $x$-coordinates of white balls are $i-\lambda_i-\frac12$ and coordinates of black ones are $\lambda'_j-j+\frac12$. Therefore $\mathbb{Z}+\frac12=\{\lambda_j'-j+\frac12\}\bigsqcup\{i-\lambda_i-\frac12\}$. Shifting by $-(n-m)-\frac12$ we get the Lemma.
\end{proof}

Recall that 
\[Q_j=\kappa_j-j+m-r=\lambda_j'-j-(n-m)\quad \text{for $1 \leq j \leq m-r$}.\] Note that $\lambda_j'-j-(n-m)\geq 0$ if and only if $1 \leq j \leq m-r$. Since $A_1,\ldots,A_r \geq 0$ and $A_i \neq Q_j$ then it follows from the Lemma \ref{lem:Z=l+l} that $\{A_i\}$ should be a subset in the set $\{i-\lambda_i-(n-m)-1| i \in \mathbb{N}\}$. Note that $i-\lambda_i-(n-m)-1 \geq 0$ if and only if  $i>n-r$. Therefore, the non-zero terms correspond to subsets $\{A_r,\ldots,A_1\} \subset \{-L_i|i>n-r\}$, where $L_i=\lambda_i-i+n-m+1$. Rewriting the determinants in \eqref{eq:chi:bos} as sums over permutations we get the following result.

\begin{Theorem}\label{th:2'} 
The generating function $\chi_{\mu,\nu,\lambda}^{n,m}(q)$ is equal to the sum
\begin{equation}
\begin{aligned}
\chi_{\mu,\nu,\lambda}^{n,m}(q)=(-1)^{r(m+n)}\!\!\!\!\! \sum_{(\sigma,\tau,A)\in \Theta}  
\!\!\!\!\!(-1)^{|\sigma|+|\tau|+\sum\limits_{i=1}^r A_i}\,\frac{q^{\Delta^{\sigma,\tau,A}(\mu,\nu,\lambda)}}{(q)_\infty^{n+m}}
\end{aligned}, \label{eq:chi:bos'}
\end{equation}
where 
\[(\sigma,\tau,A)\in \Theta \Leftrightarrow \sigma \in S_n, \tau \in S_m, A_{r-i+1}=-L_{s_i}, \text{ for } s_r>\cdots>s_1>n-r, \]
and 
\begin{multline*}
\Delta^{\sigma,\tau,A}(\mu,\nu,\lambda)=\sum\limits_{i=1}^r A_i \left(\frac{A_i+1}2+N_{\sigma(i)}-M_{\tau(i)}\right) 
\\ -\sum\limits_{i=r+1}^n (P_{i-r}+1)(N_{\sigma(i)}-N_{i-r})- \sum\limits_{i=r+1}^m Q_{i-r}(M_{\tau(i)}-M_{i-r}).
\end{multline*}
\end{Theorem}

This theorem will be proven in Section \ref{sec:Brion}.
In this proof we consider inequalities $a_{i,j}\geq a_{i+1,j}$, $a_{i,j}\geq a_{i,j+1}$ and conditions \eqref{eq:asymp} as a definition  of a polyhedron (infinite dimensional) and the generating function $\chi_{\mu,\nu,\lambda}^{n,m}(q)$ as a sum over integer points in the polyhedron. This sum is calculated using Brion theorem~\cite{Brion}. Each term in \eqref{eq:chi:bos'} corresponds to a vertex contribution in Brion theorem.

%

\section{Lattice paths} \label{sec:Paths}

\subsection{\!\!} Let $G$ be an oriented graph with the set of vertices  $V$ and the set of edges $E$. To any edge $e \in E$ we assign a weight $w(e)$. For any path $p=(e_1,e_2,\ldots,e_n)$ we define the weight as a product of the edge weights $w(p)=\prod w(e_i)$. 

For any two vertices $s,t$ we denote $P(s\rto t)=\sum_p w(p)$, where the summation goes over all paths from $s$ to $t$. Below we will assume that $P(s\rto t)$ is well-defined. Usually this follows from the condition that number of paths from $s$ to $t$ is finite (for example in \cite{Aigner} this follows from the conditions that $G$ is finite and has no oriented cycles). In our case the weight of the edge $w(e)\in \{q^{\mathbb{Z}_{\geq 0}}\}$, where $q$ is a formal variable, and we assume that for any fixed $H\in \mathbb{Z}_{\geq 0}$ number of paths from $s$ to $t$ of the weight $q^H$ is finite. Then $P(s\rto t)$ is well-defined as a formal series in $q$, $P(s\rto t)\in \mathbb{C}[[q]]$. 

For any sets of $n$ source vertices $S=\{s_1,\ldots,s_n\}$ and $n$ target vertices $T=\{t_1,\ldots,t_n\}$ we denote $P(S\rto T)=\sum_{p_1,\ldots,p_n} w(p_1)\cdot\ldots\cdot w(p_n)$, where the summation goes over all sets of paths such that $p_i$ goes from $s_i$ to $t_i$. Clearly $P(S\rto T)=P(s_1 \rto t_1)\cdot\ldots\cdot P(s_n\rto t_n)$.

By $P_{nc}(S\rto T)$ we denote the sum  $\sum_{p_1,\ldots,p_n} w(p_1)\cdot\ldots\cdot w(p_n)$ where set of paths is assumed to be without crossings. The Lindstr\"{o}m-Gessel-Viennot lemma provides an efficient way to find $P_{nc}(S\rto T)$. Standard references for this lemma are \cite{Lind}, \cite{Gessel Viennot}, for the clear introduction see \cite{Aigner}.

\begin{Lemma*}[Lindstr\"{o}m-Gessel-Viennot] 
For an oriented graph $G$ as above and any sets of sources and targets $S=\{s_1,\ldots,s_n\}, T=\{t_1,\ldots,t_n\}$ we have
\[\sum_{\sigma \in S_n}(-1)^{|\sigma|}P_{nc}(S\rto \sigma(T))=\det\begin{pmatrix}
P(s_i \rto t_j)
\end{pmatrix}_{i,j=1}^n  \]
\end{Lemma*}
In most examples (and in all examples in this paper) $P_{nc}(S\rto \sigma(T))\neq 0$ for only one permutation $\sigma$. 
In this case 
\[P_{nc}(S\rto \sigma(T))=(-1)^{|\sigma|}\det\begin{pmatrix}
P(s_i \rto t_j)
\end{pmatrix}_{i,j=1}^n.\]

\subsection{\!\!}
In this paper we use graph $G$ with vertices $(a+\frac12,b)$, where $a,b\in \mathbb{Z}, b\geq 0$. There are two types of edges namely the horizontal ones $(a+\frac12,b) \rightarrow (a+\frac32,b)$ ($\rightarrow$ denotes orientation) and vertical ones $(a+\frac12,b) \rightarrow (a+\frac12,b+1)$ for $a<0$ and $(a+\frac12,b) \leftarrow (a+\frac12,b+1)$ for $a\geq 0$. The weight of a vertical edge is 1, the weight of a horizontal edge on the line $y=b$ is $q^{b}$. 

Note that the number of paths from $s=(\frac{1}{2},b)$ to $t=(\frac{1}{2}+a,0)$, $a,b \geq 0$ is equal to the binomial coefficient $\binom{a+b}b$. The number of paths counted with weights is equal to the $q$-binomial coefficient $P(s \rto t)=\qbinom{a+b}{b}q$.

We will use ``infinitely remote'' source and target vertices, see an example in Fig. \ref{fig:paths:m-0}. We say that a path starts at the point $(-\infty,b)$ if the path contains all sufficiently left edges on the horizontal line $y=b$. Similarly we define paths which start at the point $(a,+\infty)$ or go to the point $(+\infty,b)$ or $(a,+\infty)$. For example, the paths from the point $s=(-\infty,0)$ to $t=(-\frac12,+\infty)$ are in one to one correspondence with Young diagrams. And in this case $P(s \rto t)$ is equal to the generating function of Young diagrams $1/(q)_\infty$.

For the ``infinitely remote'' source and target vertices we need to define the weight of the path. The problem happens for the vertices $(-\infty,b)$ since their paths contain infinitely many horizontal edges on the line $y=b$ and therefore the weights of these paths are not defined. We divide by $q^{b}$ the weight of each horizontal edge (of such paths) over the point $(i,0)$, $i<0$. Clearly there is no more than one such edge, if there is none we just divide the weight of the path by $q^b$. Informally speaking, we assign the weight $q^{-b(\infty/2-1/2)}$ to the vertex $(-\infty,b)$. For example, for $s=(-\infty,b)$, $t=(-a-\frac12,+\infty)$, $a,b \geq 0$ we have $P(s \rto t)=q^{-ab}/(q)_\infty$.

For the $(+\infty,b)$ we divide by $q^{b}$ the weight of each horizontal edge (of path to the $(+\infty,b)$) over the point $(i,0)$, $i\geq 0$.  Informally speaking we assign the weight $q^{-b(\infty/2+1/2)}$ to the vertex $(+\infty,b)$. For example, for $s=(a+\frac12,+\infty)$, $t=(+\infty,b)$, $a,b \geq 0$ we have $P(s \rto t)=q^{-(a+1)b}/(q)_\infty$.

Now we prove the formula  \eqref{eq:Wn:char} for the number of plane partitions  with $n$ rows and asymptotic conditions.

\begin{Proposition}\label{pr:m=0LGV} 
The generating function of plane partitions $\{a_{i,j}\}$, such that $1 \leq i \leq n$, $j \in \mathbb{N}$, $a_{i,j}=\infty$ iff $(i,j) \in \lambda$, $\lim_{j \rightarrow \infty} a_{ij}=\nu_i$ has the form 
\[
\chi_{\varnothing,\nu,\lambda}^{n,0}(q) =
\frac{q^{\sum_{i=1}^n (\lambda_i+n-i)(\nu_{i}+n-i) }}{(q)_\infty^n} a_{\nu+\rho}(q^{-\lambda-\rho}).
\]
\end{Proposition}
\begin{proof} There is a natural bijection between such plane partitions and collections of non-intersecting paths from $S=\{s_1,\ldots,s_n\}$, $s_i=(\lambda_i+n-i+\frac12,+\infty)$ to $T=\{t_1,\ldots,t_n\}$, $t_i=(+\infty,\nu_i+n-i)$. The first row of the plane partition encodes the path from $s_1$ to $t_1$, the second row of the plane partition encodes the path from $s_2$ to $t_2$ and so on. The coordinates of the sources and targets are specified in such a way that plane partition condition $a_{i,j}\geq a_{i+1,j}$ is equivalent to the non-intersection of paths. 

In the Fig. \ref{fig:paths:m-0} we give an example, where $n=3$, $\lambda=(2,1,1)$, $\nu=(3,1,1)$.

\begin{figure}[h]
\begin{tikzpicture}[scale=.5,font=\small,baseline={([yshift=-.5ex]current bounding box.center)}]
\draw[thick] (11,0)--(2,0)--(2,-1)--(1,-1)--(1,-3)--(11,-3);
\draw[help lines] (0,0) grid (11,-3);
\foreach \m [count=\y] in {%
    {\infty,\infty,3,3,3,3,3,3,3,\dots,3},
    {\infty,3,3,2,1,1,1,1,1,\dots,1},
    {\infty,3,3,1,1,1,1,1,1,\dots,1},
  } {
  \foreach \n [count=\x] in \m {
     \node at (\x-.5,-\y+.5) {$\n$};
  }
}
\end{tikzpicture}
\quad$\longleftrightarrow$\quad
\begin{tikzpicture}[x=1.5em,y=1.5em,font=\small,baseline={([yshift=-.5ex]current bounding box.center)}]
\draw[step=1.5em,gray, dotted] (3.5,-0.5) grid (14.5,7.5);
\draw[fill] (7,8) circle (2pt) node [anchor=south east] {$s_3$};
\draw[fill] (8,8) circle (2pt) node [anchor=south east] {$s_2$};
\draw[fill] (10,8) circle (2pt) node [anchor=south east] {$s_1$};
\draw[fill] (15,1) circle (2pt) node [anchor=north west] {$t_3$};
\draw[fill] (15,2) circle (2pt) node [anchor=north west] {$t_2$};
\draw[fill] (15,5) circle (2pt) node [anchor=north west] {$t_1$};
\draw[-,thick] (7,7.5) -- (7,3) -- (9,3) -- (9,1) -- (14.5,1);
\draw[-,thick] (8,7.5) -- (8,4) -- (10,4) -- (10,3) -- (11,3) -- (11,2) -- (14.5,2);
\draw[-,thick] (10,7.5) -- (10,5) --  (14.5,5);
\draw[dashed] (3.5,0) -- (14.5,0);
\draw[dashed] (5.5,-0.5) -- (5.5,7.7);
\end{tikzpicture}
\caption{\label{fig:paths:m-0}}
\end{figure}

As was noted before we have $P(s_i \rto t_j)=q^{-(\lambda_i+n-i+1)(\nu_j+n-j)}/(q)_\infty$. Therefore using the Lindstr\"{o}m-Gessel-Viennot lemma we get 
\[
P(S\rto T)=\det\begin{pmatrix}
\dfrac{q^{-(\lambda_i+n-i+1)(\nu_j+n-j)}}{(q)_\infty}
\end{pmatrix}.
\]
Note that the function $P(S\rto T)$ differs from $\chi_{\varnothing,\nu,\lambda}^{n,0}(q)$ by a certain power of $q$ since our grading on paths differs slightly from the definition \eqref{eq:pp:grad}. In particular, $\chi_{\varnothing,\nu,\lambda}^{n,0}(q)$ has the leading term 1, but $P(S\rto T)$ has the leading term $q^{-\sum_i(\lambda_i+n-i+1)(\nu_i+n-i)}$. Multiplying $P(S\rto T)$ by $q^{\sum_i(\lambda_i+n-i+1)(\nu_i+n-i)}$ we get Proposition \ref{pr:m=0LGV}. 
\end{proof}
\subsection{\!\!} We have not discussed one type of paths between ``infinitely remote'' vertices. Namely let $s=(-\infty,b)$, $t=(+\infty,a)$. Then the paths from $s$ to $t$ are in one-to-one correspondence with $V$-partitions with asymptotic conditions $\lim_{i\rightarrow \infty} a_i=a$, $\lim_{i\rightarrow \infty} b_i=b$, see Fig. \ref{fig:paths:R}.

\begin{figure}[h]
\begin{tikzpicture}[x=1.5em, y=1.5em, font = \small]
\draw (0,0.5) node [anchor=east] {$\left(5\;\; \begin{matrix}
4 & 4 & 2 & 2 & \ldots\\
3 & 3 & 1 & 1 &\ldots
\end{matrix} \right)$ \qquad $\longleftrightarrow$ \qquad };
\draw[step=1.5em,gray, dotted] (1.5,-2.5) grid (13.5,4.5);
\draw[fill] (1,-1) circle (2pt) node [anchor=south east] {$s$};
\draw[fill] (14,0) circle (2pt) node [anchor=north west] {$t$};
\draw[-,thick] (1.5,-1) -- (5,-1) -- (5,1) -- (7,1) -- (7,3) -- (8,3) -- (8,2) -- (10,2) -- (10,0) -- (13.5,0);
\draw[dashed] (1.2,-2) -- (13.8,-2);
\draw[dashed] (7.5,-2.5) -- (7.5,4.7);
\end{tikzpicture}
\caption{\label{fig:paths:R}}
\end{figure}

Recall that the generating function of $V$-partitions was given in Lemma \ref{lem:V(a,b)} and equals $R(a-b;q)$. Now we are ready to prove Theorem \ref{th:1}.

\begin{proof}[Proof of Theorem \ref{th:1}]
First, we use a one-to-one correspondence between plane partitions satisfying \eqref{eq:asymp},\eqref{eq:pit} and certain lattice paths. We decompose the base of plane partition into $r$ infinite hooks, $m-r$ infinite columns and $n-r$ infinite rows. 

We set the sources and targets to be the points
\begin{align*}
s_i=&\left\{ \begin{aligned} 
&(-\infty,M_i)\;& \text{for } 1\leq &i \leq m,
\\  
&(P_{i-m}+\frac12,+\infty)\;& \text{for } m+1 \leq &i \leq m+n-r\end{aligned} 
\right.,
\\
t_j=&\left\{ \begin{aligned} &(+\infty,N_j)\; &\text{for } 1\leq &j \leq n,\\ &(-Q_{j-n}-\frac12,+\infty)\; &\text{for } n+1 \leq &j \leq m+n-r
\end{aligned}\right.. 
\end{align*}

We illustrate the correspondence in the Fig. \ref{fig:paths:mngeneric}, where  we have $n=3$, $m=2$, $\lambda=(2,1,1)$, $\nu=(3,1,1)$, $\mu=(2,0)$. By previous definitions $r=1$, $\pi=(1,0)$, $\kappa=(1)$, $N_1=5$, $N_2=2$, $N_3=1$, $M_1=3$, $M_2=0$, $P_1=2$, $P_2=0$, $Q_1=1$.

\begin{figure}[h]
\begin{tikzpicture}[scale=.5,font=\small,baseline={([yshift=-.5ex]current bounding box.center)}]
\draw[thick] (0,-9)--(0,-3)--(1,-3)--(1,-1)--(2,-1)--(2,0)--(11,0) (11,-3)--(2,-3)--(2,-9);
\clip (0,0)--(11,0)--(11,-3)--(2,-3)--(2,-9)--(0,-9)--cycle;
\draw[help lines] (0,0) grid (11,-9);
\foreach \m [count=\y] in {%
    {\infty,\infty,3,3,3,3,3,3,3,\dots,3},
    {\infty,3,3,2,1,1,1,1,1,\dots,1},
    {\infty,3,3,1,1,1,1,1,1,\dots,1},
    {4,2},
    {3,1},
    {2,0},
    {2,0},
    {\dots,\dots},
    {2,0},
  } {
  \foreach \n [count=\x] in \m {
     \node at (\x-.5,-\y+.5) {$\n$};
  }
}
\end{tikzpicture}
$\longleftrightarrow$ 
\begin{tikzpicture}[scale=.9,x=1.5em,y=1.5em,font=\small,baseline={([yshift=-.5ex]current bounding box.center)}]
\draw[step=1.5em,gray, dotted] (1.5,-0.5) grid (14.5,8.5);
\draw[fill] (1,0) circle (2pt) node [anchor=south east] {$s_2$};
\draw[fill] (1,3) circle (2pt) node [anchor=south east] {$s_1$};
\draw[fill] (8,9) circle (2pt) node [anchor=south east] {$s_4$};
\draw[fill] (10,9) circle (2pt) node [anchor=south east] {$s_3$};
\draw[fill] (6,9) circle (2pt) node [anchor=south east] {$t_4$};
\draw[fill] (15,1) circle (2pt) node [anchor=north west] {$t_3$};
\draw[fill] (15,2) circle (2pt) node [anchor=north west] {$t_2$};
\draw[fill] (15,5) circle (2pt) node [anchor=north west] {$t_1$};
\draw[-,thick] (1.5,0) -- (5,0)  -- (5,1) -- (6,1) -- (6,2) -- (7,2) -- (7,3) -- (9,3) -- (9,1) --(14.5,1);
\draw[-,thick] (1.5,3) -- (4,3)  -- (4,4) -- (5,4) -- (5,5) -- (6,5) -- (6,8.5) ;
\draw[-,thick] (8,8.5) -- (8,4) -- (10,4) -- (10,3) -- (11,3) -- (11,2) -- (14.5,2);
\draw[-,thick] (10,8.5) -- (10,5) --  (14.5,5);
\draw[dashed] (1.5,0) -- (14.5,0);
\draw[dashed] (7.5,-0.7) -- (7.5,8.7);
\end{tikzpicture}
\caption{\label{fig:paths:mngeneric}}
\end{figure}

Due to our order of $s_i$ and $t_j$ non-intersecting paths correspond to the permutation
\[\begin{pmatrix}
s_1 & s_2 & \ldots & s_{m-r} & s_{m-r+1} & \ldots  & s_m & s_{m+1} & \ldots & s_{m+n-r-1} & s_{m+n-r} 
\\
t_{n+1} & t_{n+2} & \ldots & t_{m+n-r} & t_{n-r+1} & \ldots  & t_n & t_{1} & \ldots & t_{n-r-1} & t_{n-r}
\end{pmatrix}
\]
We denote this permutation of indexes $1,\ldots,m+n-r$ by $\sigma_{m,n,r}$.
So we proved that
\[\chi_{\mu,\nu,\lambda}^{n,m}(q)=q^{\ldots} P(S \rto \sigma_{m,n,r}(T)).\]

We compute the value $P(S \rto \sigma_{m,n,r}(T))$ from the Lindstr\"{o}m-Gessel-Viennot lemma. The number of inversions in the permutation $\sigma_{m,n,r}$ is  equal to  $mn-r^2$.
It remains to evaulate $P(s_i \rto t_j)$, which have been actually found above
\[P(s_i \rto t_j)=\left\{ \begin{aligned} 
&R(N_j-M_i;q)\; &\text{for } &1\leq i \leq m,\, &1\leq j \leq n,\\ 
&q^{-M_iQ_{j-n}}/(q)_\infty \; &\text{for } &1\leq i \leq m,\, &n+1 \leq j \leq m+n-r\\ 
&q^{-N_j(P_{i-m}+1)}/(q)_\infty \; &\text{for } m+1\leq &i \leq m+n-r,\, &1 \leq j \leq n\\
&0 \; &\text{for } m+1\leq &i \leq m+n-r,\, &n+1 \leq j \leq m+n-r\end{aligned}
\right..\] 
Combining all together we obtain Theorem \ref{th:1}. As above, the additional factor $q^{\sum_{j=1}^{m-r}{M_jQ_j}+\sum_{i=1}^{n-r} N_i(P_i+1)}$ comes from the difference between definition of grading in terms of paths and \eqref{eq:pp:grad}.
\end{proof}
\subsection{\!\!} \label{ssec:LGV:lem} In this Subsection we prove Lemma \ref{lem:LGV=bos}.

\begin{proof}
We want to calculate the determinant of the matrix 
\[
\mathrm{M}=\begin{pmatrix}
\left(\sum_{a\geq 0}(-1)^a q^{\binom{a+1}2}q^{(N_j-M_i)a}\right)_{\substack{1 \leq i \leq m\\ 1 \leq j \leq n}} & 
\left(q^{-M_iQ_{j}}\right)_{\substack{1 \leq i \leq m\\ 1 \leq j \leq m-r}}\\ 
\left(q^{-N_j(P_{i}+1)}\right)_{\substack{1 \leq i \leq n-r\\ 1 \leq j \leq n}}  & 0
\end{pmatrix}.\]
Denote 
\[
\widetilde{\mathrm{M}}=\begin{pmatrix}
\left(\sum_{a\geq 0}(-1)^a q^{\binom{a+1}2}q^{(N_j-M_i)a}\right)_{\substack{1 \leq i \leq m\\ 1 \leq j \leq n}} & 
\left((-1)^{Q_j}q^{\binom{Q_j+1}2}q^{-M_iQ_{j}}\right)_{\substack{1 \leq i \leq m\\ 1 \leq j \leq m-r}}\\ 
\left(q^{-N_j(P_{i}+1)}\right)_{\substack{1 \leq i \leq n-r\\ 1 \leq j \leq n}}  & 0
\end{pmatrix}.\]
Clearly, we have $\det\mathrm{M}=\mathrm{C}\det \widetilde{\mathrm{M}}$, where 
$\mathrm{C}=(-1)^{\sum Q_j}q^{-\sum \binom{Q_j+1}2}$. We decompose the matrix $\widetilde{\mathrm{M}}$ as a product of two (infinite) matrices
\[ 
\widetilde{\mathrm{M}}=\begin{pmatrix}
0 & \left((-1)^aq^{\binom{a+1}2-aM_i}\right)_{\substack{1\leq i \leq m, \\ a \geq 0}}\\ 
\left(\delta_{-a-1,P_i} \right)_{\substack{1\leq i \leq n-r,\\ a <0 }} & 0
\end{pmatrix}
\begin{pmatrix}
 \left(q^{aN_j}\right)_{\substack{a \in \mathbb{Z}, \\ 1 \leq j \leq n}} 
& \left((-1)^{Q_j}\delta_{a,Q_j} \right)_{\substack{a \in \mathbb{Z},\\  1\leq j \leq m-r}}
\end{pmatrix},
\]
Now we apply the Cauchy--Binet formula, the numbers of columns in the minor from the first factor (the numbers of rows in the minor from the second factor) we denote by $-P_1-1,\ldots,-P_{n-r}-1, A_1,\ldots,A_r, Q_1,\ldots, Q_{m-r},$ and get
\[
\det \mathrm{M}=(-1)^{(m-r)(n-r)}\sum_{A_1>A_2>\ldots>A_r\geq 0} \!\!\!\!\!\! (-1)^{\sum\limits_{i=1}^r A_i}q^{\sum\limits_{i=1}^r \binom{A_i+1}2} a_{N}(q^A,q^{-P-1})a_{M}(q^{-A},q^{-Q})
\]
In this sum we a-priori have extra condition  $A_i \neq Q_j$ but we ignore it since the corresponding terms vanish.

Therefore we proved that the right sides of \eqref{eq:chi:LGV} and \eqref{eq:chi:bos} are equal.
\end{proof}

\section{Integer Points in Polyhedra} \label{sec:Brion}

\subsection{\!\!} In this section we give a combinatorial proof of Theorem \ref{th:2'}. The proof is based on Brion's theorem which we briefly recall. Standard references for this theorem are \cite{Brion}, \cite{KhPu1}, \cite{KhPu2}, \cite{Lawr}, for a clear introduction see, for example, \cite{Beck Haase}.

Let $P \subset \mathbb{R}^N$ be a convex polyhedron, i.e., an intersection of a finite number of half-spaces. Note that $P$ can be unbounded. For simplicity we assume below that vertices of $P$ have integer coordinates and edges have rational directions. 

For a point $p=(p_1,\ldots,p_N)\in \mathbb{Z}^N$ by $t^p$ we denote $t_1^{p_1}\cdots t_N^{p_N}$, where $t_1,\ldots,t_N$ are formal variables. Define the characteristic function of  $P$ by the formula
\[
S(P)= \sum_{p \in P \cap \mathbb{Z}^n}t^p.
\]
In this definition $S(P)$ is a formal series, $S(P)\in \mathbb{Z}[[t_1^{\pm1},\ldots, t_N^{\pm1}]]$. It can be proven that there exist two Laurent polynomials $f,g \in \mathbb{Z}[t_1^{\pm1},\ldots, t_N^{\pm1}]$ such that $f S(P)=g$ (see e.g. \cite[Theorem 13.8]{Barvinok}). We denote $\mathbb{S}(P)=g/f \in \mathbb{Q}(t_1,\ldots,t_n)$. Clearly $\mathbb{S}(P)$ does not depend on the particular choice of $f,g \in \mathbb{Z}[t_1^{\pm1},\ldots, t_N^{\pm1}]$ 

For any vertex $v \in P$, we denote by $\mathcal{K}_v$ its cone, i.e., the intersection of half-spaces corresponding to the facets (maximal proper faces) of $P$ containing
$v$.

\begin{Theorem*}[Brion] For any convex polyhedron $P$ with integer vertices and rational directions of edges we have
\[ \mathbb{S}(P)=\sum_{v} \mathbb{S}(\mathcal{K}_v).\]
\end{Theorem*}

Plane partitions satisfying \eqref{eq:pit} and \eqref{eq:asymp} are integer points of the polyhedron $P^{n,m}_{\mu,\nu,\lambda}$ in the space with coordinates $t_{i,j}$ $(i,j)\in \mathbb{N}^2\setminus \lambda$. The polyhedron $P^{n,m}_{\mu,\nu,\lambda}$ is defined by the inequalities
\begin{equation}\label{eq:Pnmlnm}
P^{n,m}_{\mu,\nu,\lambda}\colon
\quad
\left\{ 
\begin{aligned}
&t_{i,j}\geq  t_{i,j+1}\; &t_{i,j}\geq  \nu_i\geq 0\\
&t_{i,j}\geq  t_{i+1,j}\; &t_{i,j}\geq  \mu_j\geq 0\\
&t_{n+1,m+1}=0 &
\end{aligned}
\right., \quad 
\end{equation}
Therefore the functions $\chi^{n,m}_{\mu,\nu,\lambda}(q)$ can be computed using Brion theorem.

Two remarks are in order. First, we stated Brion theorem for finite dimensional polyhedra, but $P^{n,m}_{\mu,\nu,\lambda}$ is infinite dimensional. Therefore we start from the finitization of $P^{n,m}_{\mu,\nu,\lambda}$, i.e., for $H, H' \in \mathbb{N}$ we 
consider the polyhedron $P^{n,m,(H,H')}_{\mu,\nu,\lambda}$  defined as 
\begin{equation}\label{eq:PnmlnmH}
P^{n,m,(H,H')}_{\mu,\nu,\lambda}\colon
\quad
\left\{ 
\begin{aligned}
&t_{i,j}\geq  t_{i,j+1}\; &t_{i,H}=\nu_i\geq 0\\
&t_{i,j}\geq  t_{i+1,j}\; &t_{H',j}=\mu_j\geq 0\\
&t_{n+1,m+1}=0 &
\end{aligned}
\right., \quad (i,j)\in \{1,\ldots,H'\}\times\{1,\ldots,H\}\setminus \lambda.
\end{equation}
We compute $\mathbb{S}(P^{n,m,(H,H')})$ and then take the limit $H,H' \rightarrow \infty$.

Second, we need a specialization of the function $\mathbb{S}(P)$ in which $t_{i,j} \rightarrow q$. We denote by $\mathbb{S}_q(P)\in \mathbb{Q}(q)$ the function obtained by composition of $\mathbb{S}$ and this specialization. The limit  $\lim_{H,H' \rightarrow \infty}q^{-\Delta^{(H,H')}}\mathbb{S}_q(P^{n,m,(H,H')}_{\mu,\nu,\lambda})$ coincides with $\chi^{n,m}_{\mu,\nu,\lambda}(q)$. Here the numbers $\Delta^{(H,H')}$ emerge due to different definitions of grading, see below. 

It will be convenient to start from a specialization $t_{i,j} \rightarrow x_{j-i+1}/x_{j-i}$. We denote by $\mathbb{S}_x(P)\in \mathbb{Q}(\{x_i\})$ composition of $\mathbb{S}$ and this specialization. Then we can set $x_i \rightarrow q^{i}$ and get $\mathbb{S}_q(P)$.

\subsection{\!\!} We explain main ideas in the case $m=0$. As the result we get a new proof of~\eqref{eq:Wn:char}. We will work under an additional assumption of strict inequalities \[\nu_1>\ldots>\nu_n ,\] in the Section \ref{ssec:nu<mu} we briefly explain (following \cite{Makhlin}) how to remove this assumption.

We start from a description of vertices of the polyhedron $P^{n,0, (H,H')}_{\varnothing,\nu,\lambda}$. Since $t_{n+1,1}=0$ one can think that indices $(i,j)$ of coordinates $t_{i,j}$ satisfy $1 \leq i \leq n$, $1 \leq j \leq H$, $(i,j) \not \in \lambda$, i.e., $(i,j)$ lies in a skew Young diagram that we will denote $(H^n)-\lambda$. Any face of our polyhedron is defined by \eqref{eq:PnmlnmH} where some of the inequalities  become equalities. For any face we construct a graph $\Gamma$ with vertices $(i,j)$ satisfying conditions above. Two vertices $(i,j)$ and $(i',j')$ are connected by an edge iff $t_{i,j}=t_{i',j'}$ for all points of the face and boxes $(i,j)$ and $(i',j')$ have a common side. 

There exist at least $n$ connected components in $\Gamma$ since $t_{i,H}=\nu_i$ and $\nu_i>\nu_j$ for $i>j$. Vertices of our polyhedron are faces of maximal codimension, i.e., corresponding to graphs having exactly $n$ connected components. See an example in Fig.~\ref{fig:graphface}. 
\begin{figure}[!h]
\begin{tikzpicture}
[scale=.72,nodes={font=\small}]
\clip[preaction={draw=black,thick}]
  (0,0)--(0,2)--(2,2)--(2,4)--
  (12,4)--(12,0)--cycle;
\draw[help lines,dashed] (0,0) grid (11.5,4);

\draw[thick]
  (.5,1.5)--(3.5,1.5)--(3.5,3.5)--(11.2,3.5)
  (2.5,1.5)--(2.5,2.5)--(4.5,2.5)--(4.5,3.5)
  (2.5,2.5)--(2.5,3.5)--(3.5,3.5)
  (.5,.5)--(4.5,.5)--(4.5,1.5)--(5.5,1.5)--(5.5,2.5)--(11.2,2.5)
  (5.5,.5)--(8.5,.5)--(8.5,1.5)--(11.2,1.5)
  (6.5,.5)--(6.5,1.5)--(9.5,1.5) (7.5,.5)--(7.5,1.5)
  (9.5,.5)--(11.2,.5)
  ;
\foreach \i in {1,...,4} {
  \foreach \j in {1,...,11} {
     \draw[fill=white] (\j-.5,4.5-\i) circle (.75mm);
  }
  \node at (11.5,{4.5-\i}) {$\nu_{\i}$};
}
\draw[thick,preaction={draw=white}]
  (0,0)--(0,2)--(2,2)--(2,4)--
  (12,4)--(12,0)--cycle;
\end{tikzpicture}
\caption{\label{fig:graphface}}
\end{figure}

Denote by $\Gamma_v$ the graph corresponding to the vertex $v$. Denote by $K_s$ connected components of~$\Gamma_v$. Each $K_s$ is a skew Young diagram. Denote by $\mathcal{K}_{s,v}$ projection of the cone $\mathcal{K}_v$ on the subspace with coordinates $t_{i,j}$ for $(i,j) \in K_s$. The inequalities defining the cone $\mathcal  K_v$ come from the equalities defining the vertex $v$ and therefore do not include $t_{i,j}$ from different connected components. Therefore the cone $\mathcal K_v$ is a product of its projections $\mathcal{K}_{s,v}$, and we have $\mathbb{S}(\mathcal{K}_v)=\prod \mathbb{S}(\mathcal{K}_{s,v})$. 

Situation simplifies since for many vertices $\mathbb{S}_q(\mathcal{K}_v)$ vanishes due to the following result.
\begin{Proposition}[{\cite[Theorem 2.1]{Makhlin}}]\label{pr:Makh}
If the connected component $K_s$ has cycles, then $\mathbb{S}_x(\mathcal{K}_{s,v})$ is equal to 0. 
\end{Proposition}
In particular, from this Proposition and calculation in acyclic case below follows that specializations $\mathbb{S}_q(\mathcal{K}_v)$ (and therefore $\mathbb{S}_q(P^{n,0,(H,H')}_{\varnothing,\nu,\lambda})$) are well-defined. Similar argument works for generic $\mathbb{S}_q(P^{n,m,(H,H')}_{\mu,\nu,\lambda})$, see the section \ref{ssec:Brion:mneq0}

Therefore by Brion theorem we have
\begin{equation} \label{eq:Pnmh=}
\mathbb{S}_q(P^{n,0, (H,H')}_{\varnothing,\nu,\lambda})=\sum_{v} \mathbb{S}_q(\mathcal{K}_v),
\end{equation}
where the summation goes over the vertices $v$ such that corresponding graphs $\Gamma_v$ are acyclic. 

Recall that a skew Young diagram $\alpha-\beta$ is called a \textit{ribbon} if it is connected and contains no $2 \times 2$ block of squares. Due to Proposition \ref{pr:Makh}, vertices of $P^{n,0, (H,H')}_{\varnothing,\nu,\lambda}$ with nonzero contribution correspond to decompositions of the skew diagram $(H^n)-\lambda$ into $n$ ribbons such that boxes $(i,H)$ belong to different ribbons.\footnote{One can compare this to Murnaghan--Nakayama rule.}

In order to discuss ribbons it is convenient to use also another combinatorial description. For any partition $\lambda$ we assign sequence of numbers $\lambda_i-i$. This map is the bijection between partitions and sequences $\{a_i|i\in \mathbb{Z}_{>0}\}$ such that  $a_i>a_{i+1}$ and $a_i=-i$ for $i>>0$. We will say that there are particles in the points $\lambda_i-i$ and holes on the other integer points. For the illustration see Fig. \ref{fig:whiteblack}, black balls are particles, white balls are holes (up to total shift by $\frac12$).

The following lemma is standard
\begin{Lemma}\label{lem:ribbonshift} Skew partition $\alpha-\beta$ is a ribbon if and only if there exist $j,k \in \mathbb{N}$ such that the set $\{\alpha_i-i\}$ is obtained from the set $\{\beta_i-i\}$ by replacement of $\beta_j-j$ by $\beta_j-j+k$.
\end{Lemma}
We will call such replacement of $\beta_j-j$ by $\beta_j-j+k$ as jump of the particle $\beta_j-j$. Therefore, in a more informal language the lemma means that a jump of one particle to the right corresponds to the addition of a ribbon to partition. The decomposition of skew partition into $n$ ribbons corresponds to the sequence of  $n$ jumps of particles. 

Assume that $H>\lambda_1+n-1$. Then the sets of particles  for partitions $(H^n)$ and $\lambda$ differ in the first $n$ numbers. Therefore to any decomposition of the skew diagram $(H^n)-\lambda$ into $n$ ribbons we assign a permutation $\sigma \in S_n$ such that corresponding jumps are from $\lambda_{i}-i$ to $H-\sigma(i)$. Since the boxes $(i,H)$ belong to different ribbons the order of jumps is unique: first jump goes to $H-1$, second to $H-2$ and so forth. On the other side to any permutation we assign the sequence of jumps due to "only if" part of the Lemma \ref{lem:ribbonshift} get a decomposition of skew diagram $(H^n)-\lambda$ into $n$ ribbons such that boxes $(i,H)$ belong to different ribbons.

So we get a one-to-one correspondence between acyclic graphs $\Gamma_v$ and permutations. For example the graph $\Gamma_v$ in the Fig. \ref{fig:permut} corresponds to the permutation~$\dbinom{1\; 2\; 3\; 4}{1\; 4\; 3\; 2}$.

\begin{figure}[!h]
\begin{tikzpicture}
[scale=.72,nodes={font=\small}]
\draw[gray,nodes={black,font=\tiny}]
  (6.2,-.2)--(1.75,4.25) node[above left={-.5mm}] {$y-x=\lambda_1-0$}
  (5.2,-.2)--(1.75,3.25) node[above left={-.5mm}] {$y-x=\lambda_2-1$}
  (2.2,-.2)--(-.25,2.25) node[above left={-.5mm}] {$y-x=\lambda_3-2$}
  (1.2,-.2)--(-.25,1.25) node[above left={-.5mm}] {$y-x=\lambda_4-3$}
;
\clip[preaction={draw=black,thick}]
  (0,0)--(0,2)--(2,2)--(2,4)--
  (12,4)--(12,0)--cycle;
\draw[help lines,dashed] (0,0) grid (11.5,4);
\draw[thick]
  (11.2,3.5)--(2.5,3.5)
  (11.2,2.5)--(2.5,2.5)--(2.5,1.5)--(.5,1.5)--(.5,.5)
  (11.2,1.5)--(3.5,1.5)--(3.5,.5)--(1.5,.5)
  (11.2,.5)--(4.5,.5)
  ;
\foreach \i in {1,...,4} {
  \foreach \j in {1,...,11} {
     \draw[fill=white] (\j-.5,4.5-\i) circle (.75mm);
  }
  \node at (11.5,{4.5-\i}) {$\nu_{\i}$};
}
\foreach \i/\j in {1/3,4/1,4/2,4/5} {
     \draw[thin,fill=gray] (\j-.5,4.5-\i) circle (1mm);
}
\draw[thick,preaction={draw=white}]
  (0,0)--(0,2)--(2,2)--(2,4)--
  (12,4)--(12,0)--cycle;
\end{tikzpicture}
\caption{\label{fig:permut}}
\end{figure}

We denote by $v_\sigma$ the acyclic vertex corresponding to $\sigma \in S_n$.  
\begin{Lemma} \label{lem:Svsigma}
For the vertex $v_\sigma$ of the polyhedron $P^{n,0, (H,H')}_{\varnothing,\nu,\lambda}$ we have
\begin{equation}\label{eq:Svsigma}
\mathbb{S}_q(\mathcal{K}_{v_\sigma})=(-1)^{|\sigma|}q^{\Delta^{\sigma,(H)}(\lambda,\nu)}/\prod_{i=1}^n(q)_{H-\sigma(i)-\lambda_i+i-1},
\end{equation}
where $(q)_k=\prod_{s=1}^k(1-q^s)$ and
\[
\Delta^{\sigma,(H)}(\lambda,\nu)=\sum_{i=1}^n(H\nu_i-i\lambda_i-i\nu_i+i^2)-\sum_{i=1}^n(\lambda_i-i)(\nu_{\sigma(i)}-\sigma(i)). 
\]
\end{Lemma}
The proof is similar to the one in \cite[Prop. 2.4]{Makhlin}.
\begin{proof} Let $K_i$ be the ribbon which corresponds to particle jump from $\lambda_i-i$ by $H-\sigma(i)$. The set $\{K_i\}$ is a set of all connected components of the graph $\Gamma_{v_\sigma}$. First, we compute the contribution of the corresponding cone $\mathbb{S}(\mathcal{K}_{i,v_\sigma})$.

Denote by $h$ the number of boxes in the ribbon $K_i$, clearly $h=H-\sigma(i)-\lambda_i+i$. We number these boxes by $1, \ldots, h$ from the bottom-left corner to the top-right corner. In order to simplify notation we denote the corresponding coordinates $t_{k,j}$ by $t_1,\ldots, t_{h}$. In the vertex $v_i$ of the cone $\mathcal{K}_{i,v_\sigma}$ all these coordinates equal $\nu_{\sigma(i)}$, therefore $\mathbb{S}_q(t^{v_i})=q^{\nu_{\sigma(i)}h}$.

The cone $\mathcal{K}_{i,v_\sigma}$ is simple. Its edges are generated by the vectors $e_1,\ldots,e_{h-1}$, where the vector $e_s$ equals $\pm (1,\ldots,1,0,\ldots,0)$, with $s$ nonzero coordinates. The sign ``$\pm$'' is equal to ``$+$'' if the box $s+1$ is a right neighbor of the box $s$ and is equal to ``$-$'' if the box $s+1$ is an upper neighbor of the box $s$. See an example in Fig. \ref{fig:edges}
\begin{figure}[h]
\begin{tikzpicture}[scale=.60,nodes={font=\tiny}]
\draw (2,-1) node {\normalsize {$e_1$}};
\draw[thick]
  (0,0)--(0,2)--(2,2)--(2,4)--(3,4)--(4,4)
  (4,3)--(3,3)--(3,1)--(1,1)--(1,0)--(0,0);
\clip
  (0,0)--(0,2)--(2,2)--(2,4)--(3,4)--(4,4)--
  (4,3)--(3,3)--(3,1)--(1,1)--(1,0)--cycle;
\draw[help lines] (0,0) grid (4,4);
\foreach \i/\j/\k in {0/0/-1} {
  \node at (\i+.5,\j+.5) {$\k$};
}
\draw[thick,preaction={draw=white}]
  (0,0)--(0,2)--(2,2)--(2,4)--(3,4)--(4,4)
  (4,3)--(3,3)--(3,1)--(1,1)--(1,0)--(0,0);
\end{tikzpicture}
\begin{tikzpicture}[scale=.60,nodes={font=\tiny}]
\draw (2,-1) node {\normalsize {$e_2$}};
\draw[thick]
  (0,0)--(0,2)--(2,2)--(2,4)--(3,4)--(4,4)
  (4,3)--(3,3)--(3,1)--(1,1)--(1,0)--(0,0);
\clip
  (0,0)--(0,2)--(2,2)--(2,4)--(3,4)--(4,4)--
  (4,3)--(3,3)--(3,1)--(1,1)--(1,0)--cycle;
\draw[help lines] (0,0) grid (4,4);
\foreach \i/\j/\k in {0/0/+1,0/1/+1} {
  \node at (\i+.5,\j+.5) {$\k$};
}
\draw[thick,preaction={draw=white}]
  (0,0)--(0,2)--(2,2)--(2,4)--(3,4)--(4,4)
  (4,3)--(3,3)--(3,1)--(1,1)--(1,0)--(0,0);
\end{tikzpicture}
\begin{tikzpicture}[scale=.60,nodes={font=\tiny}]
\draw (2,-1) node {\normalsize {$e_3$}};
\draw[thick]
  (0,0)--(0,2)--(2,2)--(2,4)--(3,4)--(4,4)
  (4,3)--(3,3)--(3,1)--(1,1)--(1,0)--(0,0);
\clip
  (0,0)--(0,2)--(2,2)--(2,4)--(3,4)--(4,4)--
  (4,3)--(3,3)--(3,1)--(1,1)--(1,0)--cycle;
\draw[help lines] (0,0) grid (4,4);
\foreach \i/\j/\k in {0/0/+1,0/1/+1,1/1/+1} {
  \node at (\i+.5,\j+.5) {$\k$};
}
\draw[thick,preaction={draw=white}]
  (0,0)--(0,2)--(2,2)--(2,4)--(3,4)--(4,4)
  (4,3)--(3,3)--(3,1)--(1,1)--(1,0)--(0,0);
\end{tikzpicture}
\begin{tikzpicture}[scale=.60,nodes={font=\tiny}]
\draw (2,-1) node {\normalsize {$e_4$}};
\draw[thick]
  (0,0)--(0,2)--(2,2)--(2,4)--(3,4)--(4,4)
  (4,3)--(3,3)--(3,1)--(1,1)--(1,0)--(0,0);
\clip
  (0,0)--(0,2)--(2,2)--(2,4)--(3,4)--(4,4)--
  (4,3)--(3,3)--(3,1)--(1,1)--(1,0)--cycle;
\draw[help lines] (0,0) grid (4,4);
\foreach \i/\j/\k in {0/0/-1,0/1/-1,1/1/-1,2/1/-1} {
  \node at (\i+.5,\j+.5) {$\k$};
}
\draw[thick,preaction={draw=white}]
  (0,0)--(0,2)--(2,2)--(2,4)--(3,4)--(4,4)
  (4,3)--(3,3)--(3,1)--(1,1)--(1,0)--(0,0);
\end{tikzpicture}
\begin{tikzpicture}[scale=.60,nodes={font=\tiny}]
\draw (2,-1) node {\normalsize {$e_5$}};
\draw[thick]
  (0,0)--(0,2)--(2,2)--(2,4)--(3,4)--(4,4)
  (4,3)--(3,3)--(3,1)--(1,1)--(1,0)--(0,0);
\clip
  (0,0)--(0,2)--(2,2)--(2,4)--(3,4)--(4,4)--
  (4,3)--(3,3)--(3,1)--(1,1)--(1,0)--cycle;
\draw[help lines] (0,0) grid (4,4);
\foreach \i/\j/\k in {0/0/-1,0/1/-1,1/1/-1,2/1/-1,2/2/-1} {
  \node at (\i+.5,\j+.5) {$\k$};
}
\draw[thick,preaction={draw=white}]
  (0,0)--(0,2)--(2,2)--(2,4)--(3,4)--(4,4)
  (4,3)--(3,3)--(3,1)--(1,1)--(1,0)--(0,0);
\end{tikzpicture}
\begin{tikzpicture}[scale=.60,nodes={font=\tiny}]
\draw (2,-1) node {\normalsize {$e_6$}};
\draw[thick]
  (0,0)--(0,2)--(2,2)--(2,4)--(3,4)--(4,4)
  (4,3)--(3,3)--(3,1)--(1,1)--(1,0)--(0,0);
\clip
  (0,0)--(0,2)--(2,2)--(2,4)--(3,4)--(4,4)--
  (4,3)--(3,3)--(3,1)--(1,1)--(1,0)--cycle;
\draw[help lines] (0,0) grid (4,4);
\foreach \i/\j/\k in {0/0/+1,0/1/+1,1/1/+1,2/1/+1,2/2/+1,2/3/+1} {
  \node at (\i+.5,\j+.5) {$\k$};
}
\draw[thick,preaction={draw=white}]
  (0,0)--(0,2)--(2,2)--(2,4)--(3,4)--(4,4)
  (4,3)--(3,3)--(3,1)--(1,1)--(1,0)--(0,0);
\end{tikzpicture}
\caption{\label{fig:edges}}
\end{figure}

Therefore $\mathbb{S}_q(\mathcal{K}_{i,v_\sigma})=q^{\nu_{\sigma(i)}h}/\prod_{s=1}^{h-1}(1-q^{\pm s})$, where the signs ``$\pm$'' were specified above. Now we want to express this product in more explicit terms. 

\begin{Lemma} \label{lem:negativeterms}
The box $s+1$ is an upper neighbor of the box $s$ if and only if there exists $j$ such that $i>j$, $\sigma(i)<\sigma(j)$ and $s=\lambda_j-j-\lambda_{i}+i$
\end{Lemma}

In terms of particles this lemma means that particle $\lambda_{i}-i$ in its jump overtakes the particle $\lambda_{j}-j$ and $s$ in lemma corresponds to overtaking place.

\begin{proof}[Proof of the Lemma]
We draw lines given by equations $y=x+c$, where $c \in \mathbb{Z}$. Such lines go through centers of boxes in our skew diagram $(H^n)-\lambda$. Such a line intersects the ribbon $K_i$ in one box if $\lambda_i-i+1\leq c \leq H-\sigma(i)$ and does not intersect otherwise.

If the box $s+1$ is an upper neighbor of the box $s$, then the right neighbor of the box $s$ belongs to another ribbon. Denote this ribbon by $K_{i'}$ and let this box (right neighbor of the box $s$) has the number $s'+1$ in the ribbon $K_{i'}$. 

If $s'>0$ then there is a previous box in the ribbon $K_{i'}$. This box has number $s'$ in $K_{i'}$ and should be the lower neighbor of the box $s'+1$. And again, the right neighbor should belong to another ribbon, denote this ribbon by $K_{i''}$ and let this box has the number $s''+1$. This process lasts until we come to the first box of a certain ribbon, denote this ribbon by $K_j$.

Note that centers of boxes $s+1$ of $K_i$, $s'+1$ of $K_{i'}$, $s''+1$ of $K_{i''}$\ldots belong to the line  given by equations $y=x+c$. Since this line goes through the center of the first box on $K_j$ we have $c=\lambda_j-j+1$. Therefore $s=\lambda_j-j-\lambda_i+i$. In such case $\sigma(i)<\sigma(j)$ since $K_j$ is below $K_i$, but $i>j$ since $s>0$. 

See the Figure \ref{fig:permut} for the demonstration of this effect.

Conversely, for any $j$ such that $j>i$ and $\sigma(j)<\sigma(i)$ consider the first box of $K_j$. The line $y=x+c$ for $c=\lambda_j-j+1$ goes through the center of this box and should intersect the ribbon $K_i$. Let the box of intersection has number $s+1$ in $K_i$. Since  $\sigma(i)<\sigma(j)$ the ribbon $K_i$ is higher then the ribbon $K_j$ and the box $s+1$ in $K_i$ is higher then the first box in the ribbon $K_j$. The box $s+1$ in $K_i$ cannot be the first box of the ribbon since $\lambda_i-i \neq \lambda_j-j$. And it is easy to prove similarly to the previous paragraphs that the box number $s$ is low neighbor of the box $s+1$. 
\end{proof}

Rewriting factors $1/(1-q^{-s})$ as $-q^{s}/(1-q^s)$ and using formula for $h$ we have 
\[
\mathbb{S}_q(\mathcal{K}_{i,v_\sigma})=(-1)^{|\{j:j<i,\sigma(j)>\sigma(i)\}|}q^{\nu_{\sigma(i)}(H-\sigma(i)-\lambda_i+i)+\sum\limits_{j<i,\sigma(j)>\sigma(i)}(\lambda_j-j-\lambda_i+i)}\!\!\!\left/\prod_{s=1}^{H-\sigma(i)-\lambda_i+i-1}(1-q^{s})\right.
\] 
Now we can find $\mathbb{S}_q(\mathcal{K}_{v_\sigma})=\prod_{i=1}^n\mathbb{S}_q(\mathcal{K}_{i,v_\sigma})$. 
Using algebraic identities
\begin{equation}\label{eq:transf}
\sum_{j<i, \sigma(j)>\sigma(i)}\Bigl(\lambda_{j}-j-\lambda_{i}+i\Bigr)=\sum_{i=1}^n(\lambda_{i}-i)(\sigma(i)-i)
\end{equation}
and
\[
\sum_{i=1}^n\nu_{\sigma(i)}(H-\sigma(i)-\lambda_{i}+i)+\sum_{i=1}^n(\lambda_{i}-i)(\sigma(i)-i)=\Delta^{\sigma,(H)}(\lambda,\nu)
\]
we get \eqref{eq:Svsigma}. \end{proof}
Now we can find $\mathbb{S}_q(P^{n,0, (H,H')}_{\varnothing,\nu,\lambda})$ using specialization of Brion theorem \eqref{eq:Pnmh=}
\[\mathbb{S}_q(P^{n,0, (H,H')}_{\varnothing,\nu,\lambda}) =\sum_{\sigma \in S_n}(-1)^{|\sigma|}q^{\Delta^{\sigma,(H)}(\lambda,\nu)}/\prod_{i=1}^n(q)_{H-\sigma(i)-\lambda_i+i-1}.\] Here we count integer points in $P^{n,0, (H,H')}_{\varnothing,\nu,\lambda}$ with the weight $q^{\sum t_{i,j}}$, which differs from the weight defined in formula \eqref{eq:pp:grad} by $q^{\Delta^{(H)}}$, where
$\Delta^{(H)}=\sum_{i=1}^n \nu_i(H-\lambda_i)$. Using identity 
\[ 
\Delta^{\sigma,(H)}(\lambda,\nu)-\Delta^{(H)}=\sum_{i=1}^n(\nu_i+n-i)(\lambda_i+n-i)-\sum_{i=1}^n(\lambda_i+n-i)(\nu_{\sigma(i)}+n-\sigma(i))
\]
we see that limit $\lim_{H \rightarrow \infty}q^{-\Delta^{(H)}}\mathbb{S}_q(P^{n,0,(H,H')}_{\varnothing,\nu,\lambda})$ coincides with the right side of \eqref{eq:Wn:char}.

\subsection{\!\!} \label{ssec:Brion:mneq0}
In this subsection we prove Theorem \ref{th:2'} under the assumption 
\begin{equation}\label{eq:nu>mu}
\nu_1>\ldots >\nu_n >\mu_1 > \ldots >\mu_m.
\end{equation}
The Theorem \ref{th:2'} is valid without this assumption and later, in subsection \ref{ssec:nu<mu}, we explain this. 

\begin{proof} As before for any vertex $v$ we construct the graph $\Gamma_v$. It follows from Proposition~\ref{pr:Makh} that  vertices with nonzero contribution in $\mathbb{S}_q$ correspond to decompositions of the skew diagram $(H^n,m^{H'-n})-\lambda$ into $m+n$ ribbons (connected components of $\Gamma_v$), where $n$ contain boxes $(i,H)$, $1 \leq i \leq n$, and $m$ contain boxes $(H',j)$, $1 \leq j \leq m$.

For any partition $\alpha$ we consider the set of particles in coordinates $\{\alpha_i-i+n-m+1\}$. Note that this set differs from the one used in Lemma~\ref{lem:ribbonshift} by $n-m+1$.

We recall notation from Section \ref{sec:MainResults}: $\{L_i=\lambda_i-i+n-m+1\}$ and  $L_i=P_i+1$, for $1 \leq i \leq n-r$, $L_i \leq 0$ for $i>n-r$. The set of particles  for $(H^n,m^{H'-n})$ equals 
\begin{equation}\label{eq:Hmparticles}
\{H{+}n{-}m,\ldots,H{-}m{+}1,0,\ldots,{-}H'{+}n{+}1,{-}H'{+}n{-}m,\ldots\}.
\end{equation}
Ribbons which contain the boxes $(i,H)$ correspond to the jumps of particles to the points $H{+}n{-}m{+}1{-}i$, $1\leq i \leq n$. Ribbons which contain the boxes $(H',j)$ correspond to the jumps from points ${-}H'{+}n{-}m{+}j$, $1 \leq j \leq m$. The order of such jumps in uniquely specified by the inequality \eqref{eq:nu>mu}\footnote{see also discussion about order of ribbons in Sec. \ref{ssec:nu<mu}}: first jumps to $H{+}n,\ldots H{+}n{-}m{+}1$ then jumps from ${-}H'{+}n{-}m{+}1,\ldots {-}H'{+}n$

Due to Lemma~\ref{lem:ribbonshift} the first $n$ jumps should start from numbers $B_1>B_2>\ldots>B_n$, $B_i\in \{L_s|s\in \mathbb{N}\}$. Assume that $H>\lambda_1+n-1$, then there are $n-r$ particles for $\lambda$ in $P_i+1$ which are not present in \eqref{eq:Hmparticles}. Therefore the numbers $P_i+1$  should belong to the set $\{B_i\}$, in other words
\begin{equation}\label{eq:B}
(B_1,B_2,\ldots,B_n)=(P_1+1\ldots,P_{n-r}+1,-A_r,\ldots,-A_1),
\end{equation}
where $A_i \in \{-L_s|s>n-r\}$. For any vertex we assign a permutation $\sigma \in S_n$ such that our $n$ ribbons replace $B_i$ by $H+n-m+1-\sigma(i)$. These data $\sigma \in S_n$ and $\{A_i|1\leq i \leq r\} \subset \{-L_s|s>n-r\}$ encode $n$ ribbons containing $(i,H)$.

Due to Lemma \ref{lem:Z=l+l} the set $\{L_i\}$ has $m-r$ non-positive holes in integers $-Q_j$. Assume that $H'>\lambda'_1+m-1$, then these holes are occupied in \eqref{eq:Hmparticles}. And adding first $n$ ribbons we have $r$ additional holes in $-A_1,\ldots,-A_r$. Therefore the last $m$ jumps should go to the numbers $-C_j$, where
\begin{equation}\label{eq:C}
(C_1,C_2,\ldots,C_m)=(Q_1\ldots,Q_{m-r},A_1,\ldots,A_r).
\end{equation}
Introduce permutation $\tau$ such that jumps go from $-H'+n-m+\tau(j)$ to $-C_j$. This permutation encodes the last $m$ jumps. In order to construct inverse map we apply "only if" part of the Lemma \ref{lem:ribbonshift} and this works only under restriction that $H'-n+m-\tau(m-r+j)>A_j$. We will call such $\tau$ \textit{admissible}. Below we go to the limit $H'\rightarrow \infty$, in this limit given set of $A_i$ do not impose any restriction on the permutation $\tau$.

We denote by $v_{\sigma,\tau,A}$ the corresponding vertex. In the example in Fig. \ref{fig:Asigmatau} we have $\sigma=\dbinom{1\; 2\; 3\; 4}{1\; 3\; 4\; 2}$, $\tau=\dbinom{1\; 2\; 3}{1\; 3\; 2}$, $A_1=2, A_2=0$.
\begin{figure}[h]
\begin{tikzpicture}
[scale=.72,nodes={font=\small}]
\draw[gray,nodes={black,font=\tiny,above left={-.5mm}}]
  (6.2,-.2)--(1.75,4.25) node {$y-x+n-m=P_1+1$}
  (5.2,-.2)--(1.75,3.25) node {$y-x+n-m=P_2+1$}
;
\draw[gray,nodes={black,font=\tiny,left={-.5mm}}]
  (3.2,-.2)--(0.75,2.25) node {$y-x+n-m=-A_2$}
  (3.2,-2.2)--(0.75, .25) node {$y-x+n-m=-A_1$}
;
\draw[gray,nodes={black,font=\tiny,right={-.5mm}}]
  (3.2,-1.2) node {$y-x+n-m=-A_2-1$} --(0.75,1.25)
  (3.2,-3.2) node {$y-x+n-m=-A_1-1$} --(0.75,-.75) 
  (3.2,-5.2) node {$y-x+n-m=-Q_1-1$} --(-.25,-1.75)
;
\clip[preaction={draw=black,thick}]
  (0,-2)--(1,-2)--(1,1)--(1,2)--(2,2)--(2,4)--
  (12,4)--(12,0)--(3,0)--(3,-7)--(0,-7)--cycle;
\draw[help lines,dashed] (0,-6.5) grid (11.5,4);
\draw[thick]
  (11.2,3.5)--(2.5,3.5)
  (11.2,2.5)--(2.5,2.5)--(2.5,1.5)--(1.5,1.5)--(1.5,-.5)
  (11.2,1.5)--(3.5,1.5)
  (11.2, .5)--(2.5,.5)
  ( .5,-6.2)--( .5,-2.5)
  (1.5,-6.2)--(1.5,-1.5)--(2.5,-1.5)--(2.5,-.5)
  (2.5,-6.2)--(2.5,-2.5)
;
\foreach \i in {1,...,10} {
  \foreach \j in {1,...,11} {
     \draw[fill=white] (\j-.5,4.5-\i) circle (.75mm);
  }
  \node at (11.5,{4.5-\i}) {$\nu_{\i}$};
  \node at ({\i-.5},-6.5) {$\mu_{\i}$};
}
\foreach \i/\j in {1/3,5/2,3/4,4/3, 7/1,5/3,7/3} {
     \draw[thin,fill=gray] (\j-.5,4.5-\i) circle (1mm);
}
\draw[thick,preaction={draw=white}]
  (0,-2)--(1,-2)--(1,1)--(1,2)--(2,2)--(2,4)--
  (12,4)--(12,0)--(3,0)--(3,-7)--(0,-7)--cycle;
\end{tikzpicture}
\caption{\label{fig:Asigmatau}}
\end{figure}

Therefore by Brion theorem we have 
\begin{equation}\label{eq:PnmHmnl=}
\mathbb{S}_q(P^{n,m, (H,H')}_{\mu,\nu,\lambda})=\sum_{\sigma,\tau,A}\mathbb{S}_q(\mathcal{K}_{v_{\sigma,\tau,A}}),
\end{equation}
where $\sigma \in S_n, A_i=-L_{s_i}, \text{ for } s_1>\cdots>s_r>n-r$ and $\tau \in S_m$ is admissible.

\begin{Lemma} For the vertex $v_{\sigma,\tau,A}$ of the polyhedron $P^{n,m, (H,H')}_{\mu,\nu,\lambda}$ we have
\begin{equation}\label{eq:SvsigmatauA}
\mathbb{S}_q(\mathcal{K}_{v_{\sigma,\tau,A}})=\frac{(-1)^{|\sigma|+|\tau|+\sum( A_i-i+1)}\;q^{\Delta^{\sigma,\tau,A,(H,H')}(\mu,\nu,\lambda)}}{\prod_{i=1}^n(q)_{H-\sigma(i)+n-m-B_i}\;\prod_{j=1}^m(q)_{H'-\tau(j)+m-n-C_j-1}},
\end{equation}
where 
\begin{multline}\label{eq:DeltasigmatauA}
\Delta^{\sigma,\tau,A,(H,H')}(\mu,\nu,\lambda)=
\sum_{i=1}^r A_i\left( 
\frac{A_i+1}{2} -N_{\sigma(n-i+1)}+M_{\tau(m-r+i)}
\right)+\sum_{i=1}^{n-r}(P_i+1)(-N_{\sigma(i)}+n-i)+\\+\sum_{j=1}^{m-r}Q_j(-M_{\tau(j)}+m-j)+\sum_{i=1}^n\nu_i(H+n-m-i+1)+\sum_{j=1}^n\mu_j(H'+n-m-j)
\end{multline}
and $N_i=\nu_i+n-i$, $M_j=\mu_j+m-j$.
\end{Lemma}
\begin{proof}
 The proof of this lemma is analogous to the one of Lemma~\ref{lem:Svsigma}. In the denominator we have a  product of $(q)_{h-1}$ where $h$ is the length of the ribbon. The power of $q$ in the numerator is made from two summands. The first one
 \[\sum_{i=1}^n\nu_{\sigma(i)}(H-\sigma(i)+(n-m)+1-B_i)+\sum_{j=1}^m\mu_{\tau(j)}(H'-\tau(j)+m-n-C_j)\]
is the sum of the weights of the vertices $v_i$ in cones $\mathcal{K}_{v_{\sigma,\tau,A},i}$ corresponding to ribbons. The second one comes from rewriting $1/(1-q^{-s})$ as $-q^{s}/(1-q^s)$. Such terms come from the edges of the form $-(1,\ldots,1,0,\ldots,0)$ as in Lemma \ref{lem:Svsigma}.

It was explained in the proof of Lemma \ref{lem:Svsigma} that for ribbons which contain boxes $(i,H)$ such terms correspond to the pairs of consecutive boxes $s$, $s+1$ such that $s+1$ is an upper neighbor of the box $s$. For ribbons which contain boxes $(H',j)$ the situation is reflected: if we number boxes from an upper right corner then such terms correspond to the pairs of consecutive boxes $s$, $s+1$ such that $s+1$ is a left neighbor of the box $s$. 

Moreover, it was explained in the proof of the Lemma \ref{lem:negativeterms} that such terms correspond to the overtaking of the particles, and the corresponding overtakings give the terms $\sum_{i=1}^nB_i(\sigma(i)-i)$ and $\sum_{j=1}^mC_j(\tau(j)-j)$ (compare with \eqref{eq:transf}). Now for ribbons which contain boxes $(i,H)$ we have an additional phenomena, namely the particle overtakes standing particles in nonpositive positions $(0,\ldots,-H+n+1)$ (except particles in position $-A_j$, for $j>i$). In terms of the proof of the Lemma \ref{lem:negativeterms} it means that the line $y=x-n+m+c$ for $c<0$, can intersect not the first box of another ribbon but vertical border of the diagram $(H^n,m^{H'-n})$. For each $1\leq i \leq r$ this gives $\sum_{s=1}^{A_i} s-\sum _{j=i+1}^r(A_i-A_j)$. The resulting formula is 
\[
 \sum_{i=1}^r\binom{A_i+1}2+\sum_{i=1}^rA_i(2i-r-1)+\sum_{i=1}^nB_i(\sigma(i)-i)+\sum_{j=1}^mC_j(\tau(j)-j).\]
 Putting all things together and using \eqref{eq:B}, \eqref{eq:C} we get \eqref{eq:DeltasigmatauA} (the sign $(-1)^{\sum A_i-i+1}$ comes from overtakings of  standing particles).
\end{proof}
Now we find $\chi^{n,m}_{\mu,\nu,\lambda}(q)$ as $\lim_{H,H' \rightarrow \infty}q^{-\Delta^{(H,H')}}\mathbb{S}_q(P^{n,m,(H,H')}_{\mu,\nu,\lambda})$, where 
\begin{multline*}
\Delta^{(H,H')}=\sum_{i=1}^{n-r}\nu_i(H-P_i-i+n-m)+\sum_{j=1}^{m-r}\mu_j(H'-Q_j-j+m-n)\\+\sum_{i=n-r+1}^{n}\nu_i(H-i+n-m+1)+\sum_{j=m-r+1}^{m}\mu_j(H'-j+m-n).
\end{multline*}
It is easy to see that $\Delta^{\sigma,\tau,A,(H,H')}(\mu,\nu,\lambda)-\Delta^{(H,H')}=\Delta^{\tilde{\sigma},\tilde{\tau},A}(\mu,\nu,\lambda)$ for  \[\tilde{\sigma}=\sigma\circ\begin{pmatrix}
1 & \ldots & r & r+1& \ldots &n\\ n & \ldots & n-r+1 & 1& \ldots & n-r
\end{pmatrix},\;\; \tilde{\tau}=\tau \circ\begin{pmatrix}
1 & \ldots & r & r+1 & \ldots & m\\ m-r+1 & \ldots & m & 1 & \ldots & m-r \end{pmatrix}
\] 
and we get formula~\eqref{eq:chi:bos'}.
\end{proof}

\subsection{\!\!} \label{ssec:nu<mu}
In the previous subsection we proved Theorem \ref{th:2'} under the assumption \eqref{eq:nu>mu}
\[
\nu_1>\ldots >\nu_n >\mu_1 > \ldots >\mu_m.
\]
But this theorem holds for any $\nu,\mu$ since it is equivalent to Theorem \ref{th:1} proven in section~\ref{sec:Paths}. In this subsection we explain how to get rid of the condition \eqref{eq:nu>mu} in the context of Brion theorem.

First, note that if some inequalities between $\nu_i,\mu_j$ become equalities then the polyhedron $P^{n,m, (H,H')}_{\mu,\nu,\lambda}$ degenerates. This degeneration changes combinatorial structure, in particular, some of the vertices merge. But one can ignore this when using Brion theorem (see arguments in \cite[Sec. 8]{Feigin Makhlin}). Therefore Theorem \ref{th:2'} still holds.

Now fix any strong order $\mathfrak{o}$ on $\{\nu_1,\ldots,\nu_n,\mu_1,\ldots, \mu_m\}$ such that $\nu_{i_1}>\nu_{i_2}$, $\mu_{j_1}>\mu_{j_2}$ for $i_1<i_2$ and $j_1<j_2$.
Proposition~\ref{pr:Makh} implies that vertices with nonzero contribution to $\mathbb{S}_q(P^{n,m, (H,H')}_{\mu,\nu,\lambda})$ correspond to decompositions of the skew diagram $(H^n,m^{H'-n})-\lambda$ into $m+n$ ribbons (of which $n$ contain boxes $(i,H)$, $1 \leq i \leq n$, and $m$ contain boxes $(H',j)$, $1 \leq j \leq m$). 

The order $\mathfrak{o}$ provides an additional condition on a such decomposition. The usual partial order on boxes ($b_i\ge b_j$ if $b_i$ lies to the northwest of $b_j$) induces a partial order on ribbons: $K_i\ge K_j$ if there are boxes $b_i\in K_i$, $b_j\in K_j$ such that $b_i\ge b_j$. Such order should be compatible with the order $\mathfrak{o}$ if we identify ribbons containing $(i,H)$ with $\nu_i$ and ribbons containing $(H',j)$ with $\mu_j$.

\begin{figure}[h]
\begin{tikzpicture}
[scale=.7,nodes={font=\small}]
\node at (4.5,-4.7) {$\nu_1>\nu_2>\mu_1>\mu_2$};
\node at (4.5,-5.3) {\phantom{$\nu_1$}};
\clip[preaction={draw=black,thick}]
  (1,2)--(2,2)--(2,4)--
  (8,4)--(8,2)--(3,2)--(3,-4)--(1,-4)--cycle;
\draw[help lines,dashed] (0,-4.5) grid (7.5,4);
\draw[thick]
  (7.2,3.5)--(2.5,3.5) 
  (7.2,2.5)--(2.5,2.5)--(2.5,1.5)--(1.5,1.5)--(1.5,.5) 
  (1.5,-3.2)--(1.5,-.5)--(2.5,-.5)--(2.5,.5) 
  (2.5,-3.2)--(2.5,-1.5) 
;
\foreach \i in {1,...,7} {
  \foreach \j in {1,...,6} {
     \draw[fill=white] (\j+.5,4.5-\i) circle (.75mm);
  }
  \node at (7.5,{4.5-\i}) {$\nu_{\i}$};
  \node at ({\i+.5},-3.5) {$\mu_{\i}$};
}
\draw[thick,preaction={draw=white}]
  (1,2)--(2,2)--(2,4)--
  (9,4)--(9,2)--(3,2)--(3,-4)--(1,-4)--cycle;
\end{tikzpicture}
\quad
\begin{tikzpicture}
[scale=.7,nodes={font=\small}]
\node at (4.5,-4.7) {$\nu_1>\mu_1>\nu_2>\mu_2,$};
\node at (4.5,-5.3) {$\nu_1>\mu_1>\mu_2>\nu_2$};
\clip[preaction={draw=black,thick}]
  (1,2)--(2,2)--(2,4)--
  (8,4)--(8,2)--(3,2)--(3,-4)--(1,-4)--cycle;
\draw[help lines,dashed] (0,-4.5) grid (7.5,4);
\draw[thick]
  (7.2,3.5)--(2.5,3.5) 
  (7.2,2.5)--(3.5,2.5) 
  (1.5,-3.2)--(1.5,1.5)--(2.5,1.5)--(2.5,2.5) 
  (2.5,-3.2)--(2.5,.5) 
;
\foreach \i in {1,...,7} {
  \foreach \j in {1,...,6} {
     \draw[fill=white] (\j+.5,4.5-\i) circle (.75mm);
  }
  \node at (7.5,{4.5-\i}) {$\nu_{\i}$};
  \node at ({\i+.5},-3.5) {$\mu_{\i}$};
}
\draw[thick,preaction={draw=white}]
  (1,2)--(2,2)--(2,4)--
  (9,4)--(9,2)--(3,2)--(3,-4)--(1,-4)--cycle;
\end{tikzpicture}
\quad
\begin{tikzpicture}
[scale=.7,nodes={font=\small}]
\node at (4.5,-4.7) {$\nu_1>\mu_1>\mu_2>\nu_2$,};
\node at (4.5,-5.3) {$\mu_1>\nu_1>\mu_2>\nu_2$};
\clip[preaction={draw=black,thick}]
  (1,2)--(2,2)--(2,4)--
  (8,4)--(8,2)--(3,2)--(3,-4)--(1,-4)--cycle;
\draw[help lines,dashed] (0,-4.5) grid (7.5,4);
\draw[thick]
  (7.2,3.5)--(2.5,3.5) 
  (7.2,2.5)--(3.5,2.5) 
  (1.5,-3.2)--(1.5,1.5) 
  (2.5,-3.2)--(2.5,2.5) 
;
\foreach \i in {1,...,7} {
  \foreach \j in {1,...,6} {
     \draw[fill=white] (\j+.5,4.5-\i) circle (.75mm);
  }
  \node at (7.5,{4.5-\i}) {$\nu_{\i}$};
  \node at ({\i+.5},-3.5) {$\mu_{\i}$};
}
\draw[thick,preaction={draw=white}]
  (1,2)--(2,2)--(2,4)--
  (9,4)--(9,2)--(3,2)--(3,-4)--(1,-4)--cycle;
\end{tikzpicture}
\caption{Some examples of decompositions and compatible orders}
\end{figure}
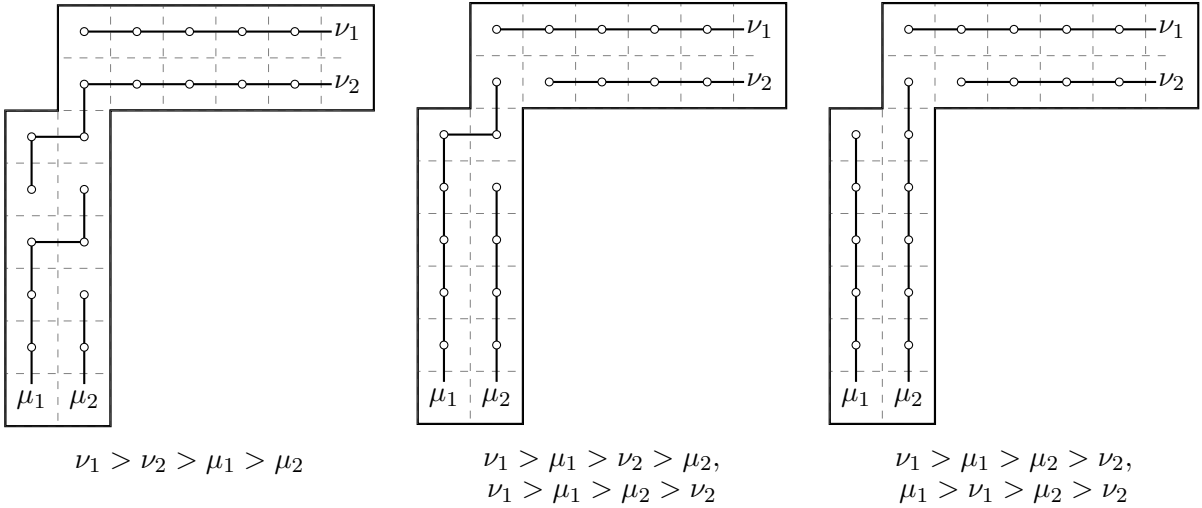

The set of vertices depends on the order $\mathfrak{o}$. But the contribution of a vertex is a product of ribbon contributions, and each such contribution is defined for any $\nu_i$ and $\mu_j$ (not necessarily compatible with $\mathfrak{o}$).

So for any given order $\mathfrak{o}$ the function $\mathbb{S}_q(P^{n,m, (H,H')}_{\mu,\nu,\lambda})$ computed using Brion theorem is defined for any nonnegative integer numbers $\nu_i,\mu_j$ not necessarily compatible with $\mathfrak{o}$. For any given $n,m,\mu,\nu,\lambda$ we denote this function just by $\mathbb{S}_{\mathfrak{o}}^{(H,H')}(q)$ and its limit (as $H,H'\rightarrow\infty$) as
$\mathbb{S}_{\mathfrak{o}}(q)$.

\begin{Example} \label{ex:n=m=1}
Let $n=m=1$, $\lambda=\varnothing$. We have two possible orders $\mathfrak{o}\colon \nu_1>\mu_1$ and $\mathfrak{o}'\colon \mu_1>\nu_1$. For $\mathfrak{o}$ we have  $H'-1$ vertices and using Brion Theorem we obtain 
\[\mathbb{S}^{(H,H')}_{\mathfrak{o}}(q)=\sum_{a=0}^{H'-2} \frac{(-1)^a q^{\binom{a+1}{2}}q^{(H+a)\nu_1+(H'-a-1)\mu_1}}{(q)_{H+a-1}(q)_{H'-a-2}}.\]
For the order $\mathfrak{o'}$ we have $H-1$ vertices and using Brion Theorem we obtain
\[\mathbb{S}^{(H,H')}_{\mathfrak{o}'}(q)=\sum_{b=0}^{H-2} \frac{(-1)^b q^{\binom{b+1}{2}}q^{(H-b-1)\nu_1+(H'+b)\mu_1}}{(q)_{H-b-2}(q)_{H'+b-1}}.\]

These formulas are different, they even have different number of summands. Brion theorem proves that the first formula calculates $\mathbb{S}_q(P^{1,1, (H,H')}_{\{\mu_1\},\{\nu_1\},\varnothing})$ for $\nu_1>\mu_1$ and the second formula for $\mu_1>\nu_1$. But in some cases both formulas work!

Namely, using the $q$-binomial theorem in the form
\[
\sum_{c=0}^{N}\frac{(-1)^c q^{\binom{c}{2}} x^c}{(q)_c(q)_{N-c}}=\frac{1}{(q)_N}\prod_{i=1}^N(1-xq^{i-1})\]
for $x=q^{\mu_1-\nu_1-H'+2}$, $N=H+H'-3$ and an obvious observation that right side vanishes for $x=q^{-d}$, $0\leq d \leq N-1$ we have 
\[\mathbb{S}^{(H,H')}_{\mathfrak{o}}(q)=\mathbb{S}^{(H,H')}_{\mathfrak{o}'}(q), \text{ for }0\leq \nu_1-\mu_1+H'-2\leq  H+H'-4.\]
It is convenient to rewrite the last inequality as 
\begin{equation}2-H'\leq \nu_1-\mu_1\leq H-2. \label{eq:inneq11}\end{equation}
In this region both formulas give the same (and therefore correct) answer.

In the limit $H,H'\rightarrow \infty$ situation becomes simpler. Indeed, for any given $\nu_1,\mu_1$ and large enough $H,H'$ the inequality \eqref{eq:inneq11} is satisfied and in the limit $\mathbb{S}_{\mathfrak{o}}(q)=\mathbb{S}_{\mathfrak{o}'}(q)$. 
\end{Example}

More generally, we have:
\begin{Proposition}\label{pr:oo'}
If $n+1-H'\leq \nu_i-\mu_j\leq H-m-1$ for any $i,j$ then for any two orders $\mathfrak{o}$, $\mathfrak{o}'$ we have 
\[\mathbb{S}^{(H,H')}_{\mathfrak{o}}(q)=\mathbb{S}^{(H,H')}_{\mathfrak{o}'}(q). \]
\end{Proposition}
\begin{proof}
It is enough to consider the case when the orders $\mathfrak{o}$, $\mathfrak{o}'$ differ only by an elementary transposition of $\nu_i$ and $\mu_j$ (in $\mathfrak{o}\colon \nu_i>\mu_j$ and $\mathfrak{o}'\colon \mu_j>\nu_i$). Recall that the summands in $\mathbb{S}^{(H,H')}_{\mathfrak{o}}(q)$ and $\mathbb{S}^{(H,H')}_{\mathfrak{o}'}(q)$ correspond to decompositions of the skew diagram $(H^n,m^{H'-n})-\lambda$ into $m+n$ ribbons compatible with the orders $\mathfrak{o}$ and $\mathfrak{o'}$ correspondingly. There are three possibilities.

\textit{Case 1.} Ribbons containing boxes $(i,H)$ and $(H',j)$ have no common edges. Corresponding summands appear both in $\mathbb{S}^{(H,H')}_{\mathfrak{o}}(q)$ and $\mathbb{S}^{(H,H')}_{\mathfrak{o}'}(q)$.

\textit{Case 2.} Union of ribbons containing boxes $(i,H)$ and $(H',j)$ is a ribbon. Denote this ribbon by $\alpha-\beta$. Fix ribbons containing other end-boxes $(i',H)$ and $(H',j')$ for $i'\neq i$ and $j'\neq j$. For such summands $\alpha-\beta$ is divided by an internal edge $e$ into two ribbons containing $(i,H)$ and $(H',j)$. Denote by $\mathbb{S}_{e, \alpha-\beta}(q)$ the product of contributions of these two ribbons. If the edge $e$ is horizontal then the corresponding term appears in $S_\mathfrak{o}^{(H,H')}(q)$, otherwise in $\mathbb{S}^{(H,H')}_{\mathfrak{o}'}(q)$. 

\begin{Lemma} \label{lem:wallsribbon}
Suppose $\alpha-\beta$ is a ribbon such that $\alpha-\beta$ lies in the rectangle $H\times H'$ and contains boxes $(i,H)$ and $(H',j)$. If $1+i-H'\leq \nu_1-\mu_1\leq H-j-1$ then
\[\sum_{e \colon \text{horizontal}}\mathbb{S}_{e, \alpha-\beta}(q)=\sum_{e\colon \text{vertical}}\mathbb{S}_{e, \alpha-\beta}(q)\]
\end{Lemma}
This lemma is a generalization of Example \ref{ex:n=m=1} and by straightforward calculation reduces to the $q$-binomial theorem. Due to this lemma the contributions to $\mathbb{S}^{(H,H')}_{\mathfrak{o}}(q)$ and $\mathbb{S}^{(H,H')}_{\mathfrak{o}'}(q)$ are equal to each other in this case.

\begin{Example} Consider $\alpha=(4,2,1,1)$ and $\beta=(1)$. In this case $\alpha-\beta$ has 6 internal edges, corresponding decompositions are drawn below in Fig. \ref{fig:exwalls}
\begin{figure}[h]
\begin{tikzpicture}[scale=.6,nodes={font=\small}]
	\node at (2.8,-3) {$\mathbb{S}_1$};
	\clip[preaction={draw=black,thick}]
	(0,-4)--(0,-1)--(1,-1)--(1,0)--(4,0)--(4,-1)--(2,-1)--(2,-2)--(1,-2) -- (1,-4)-- cycle;
	\draw[help lines,dashed] (0,-4) grid (4,0);
	\node at (3.5,-0.5) {$\nu_{1}$};	
	\node at (0.5,-3.5) {$\mu_{1}$};			
	\draw[thick]
	(3.2,-0.5)--(1.5,-0.5) -- (1.5,-1.5) -- (0.5,-1.5) -- (0.5,-2.5) 
	;
	\foreach \i in {1,...,3} {
		\foreach \j in {0,...,2} {
			\draw[fill=white] (\j+.5,0.5-\i) circle (.75mm);
		}
	}
	\draw[thick,preaction={draw=white}]
	(0,-4)--(0,-1)--(1,-1)--(1,0)--(4,0)--(4,-1)--(2,-1)--(2,-2)--(1,-2) -- (1,-4)-- cycle;
\end{tikzpicture} 
\begin{tikzpicture}[scale=.6,nodes={font=\small}]
	\node at (2.8,-3) {$\mathbb{S}_2$};
	\clip[preaction={draw=black,thick}]
	(0,-4)--(0,-1)--(1,-1)--(1,0)--(4,0)--(4,-1)--(2,-1)--(2,-2)--(1,-2) -- (1,-4)-- cycle;
	\draw[help lines,dashed] (0,-4) grid (4,0);
	\node at (3.5,-0.5) {$\nu_{1}$};	
	\node at (0.5,-3.5) {$\mu_{1}$};			
	\draw[thick]
	(3.2,-0.5)--(1.5,-0.5) -- (1.5,-1.5) -- (0.5,-1.5) 
	(0.5,-3.2)--(0.5,-2.5) 
	;
	\foreach \i in {1,...,3} {
		\foreach \j in {0,...,2} {
			\draw[fill=white] (\j+.5,0.5-\i) circle (.75mm);
		}
	}
	\draw[thick,preaction={draw=white}]
	(0,-4)--(0,-1)--(1,-1)--(1,0)--(4,0)--(4,-1)--(2,-1)--(2,-2)--(1,-2) -- (1,-4)-- cycle;
\end{tikzpicture}
\begin{tikzpicture}[scale=.6,nodes={font=\small}]
	\node at (2.8,-3) {$\mathbb{S}_3$};
	\clip[preaction={draw=black,thick}]
	(0,-4)--(0,-1)--(1,-1)--(1,0)--(4,0)--(4,-1)--(2,-1)--(2,-2)--(1,-2) -- (1,-4)-- cycle;
	\draw[help lines,dashed] (0,-4) grid (4,0);
	\node at (3.5,-0.5) {$\nu_{1}$};	
	\node at (0.5,-3.5) {$\mu_{1}$};			
	\draw[thick]
	(3.2,-0.5)--(1.5,-0.5) -- (1.5,-1.5) 
	(0.5,-3.2)--(0.5,-1.5) 
	;
	\foreach \i in {1,...,3} {
		\foreach \j in {0,...,2} {
			\draw[fill=white] (\j+.5,0.5-\i) circle (.75mm);
		}
	}
	\draw[thick,preaction={draw=white}]
	(0,-4)--(0,-1)--(1,-1)--(1,0)--(4,0)--(4,-1)--(2,-1)--(2,-2)--(1,-2) -- (1,-4)-- cycle;
\end{tikzpicture}
\begin{tikzpicture}[scale=.6,nodes={font=\small}]
	\node at (2.8,-3) {$\mathbb{S}_4$};
	\clip[preaction={draw=black,thick}]
	(0,-4)--(0,-1)--(1,-1)--(1,0)--(4,0)--(4,-1)--(2,-1)--(2,-2)--(1,-2) -- (1,-4)-- cycle;
	\draw[help lines,dashed] (0,-4) grid (4,0);
	\node at (3.5,-0.5) {$\nu_{1}$};	
	\node at (0.5,-3.5) {$\mu_{1}$};			
	\draw[thick]
	(3.2,-0.5)--(1.5,-0.5)  
	(0.5,-3.2)--(0.5,-1.5) -- (1.5,-1.5) 
	;
	\foreach \i in {1,...,3} {
		\foreach \j in {0,...,2} {
			\draw[fill=white] (\j+.5,0.5-\i) circle (.75mm);
		}
	}
	\draw[thick,preaction={draw=white}]
	(0,-4)--(0,-1)--(1,-1)--(1,0)--(4,0)--(4,-1)--(2,-1)--(2,-2)--(1,-2) -- (1,-4)-- cycle;
\end{tikzpicture}\;
	\begin{tikzpicture}[scale=.6,nodes={font=\small}]
	\node at (2.8,-3) {$\mathbb{S}_5$};
	\clip[preaction={draw=black,thick}]
	(0,-4)--(0,-1)--(1,-1)--(1,0)--(4,0)--(4,-1)--(2,-1)--(2,-2)--(1,-2) -- (1,-4)-- cycle;
	\draw[help lines,dashed] (0,-4) grid (4,0);
	\node at (3.5,-0.5) {$\nu_{1}$};	
	\node at (0.5,-3.5) {$\mu_{1}$};			
	\draw[thick]
	(3.2,-0.5)  --(2.5,-0.5) 
	(0.5,-3.2)--(0.5,-1.5) -- (1.5,-1.5) --(1.5,-0.5) 
	;
	\foreach \i in {1,...,3} {
		\foreach \j in {0,...,2} {
			\draw[fill=white] (\j+.5,0.5-\i) circle (.75mm);
		}
	}
	\draw[thick,preaction={draw=white}]
	(0,-4)--(0,-1)--(1,-1)--(1,0)--(4,0)--(4,-1)--(2,-1)--(2,-2)--(1,-2) -- (1,-4)-- cycle;
\end{tikzpicture}
\begin{tikzpicture}[scale=.6,nodes={font=\small}]
	\node at (2.8,-3) {$\mathbb{S}_6$};				
	\clip[preaction={draw=black,thick}]
	(0,-4)--(0,-1)--(1,-1)--(1,0)--(4,0)--(4,-1)--(2,-1)--(2,-2)--(1,-2) -- (1,-4)-- cycle;
	\draw[help lines,dashed] (0,-4) grid (4,0);
	\node at (3.5,-0.5) {$\nu_{1}$};	
	\node at (0.5,-3.5) {$\mu_{1}$};			
	\draw[thick]
	(3.2,-0.5) 
	(0.5,-3.2)--(0.5,-1.5) -- (1.5,-1.5) --(1.5,-0.5) -- (2.5,-0.5) 
	;
	\foreach \i in {1,...,3} {
		\foreach \j in {0,...,2} {
			\draw[fill=white] (\j+.5,0.5-\i) circle (.75mm);
		}
	}
	\draw[thick,preaction={draw=white}]
	(0,-4)--(0,-1)--(1,-1)--(1,0)--(4,0)--(4,-1)--(2,-1)--(2,-2)--(1,-2) -- (1,-4)-- cycle;
\end{tikzpicture}
\caption{}\label{fig:exwalls}
\end{figure}	
The terms $\mathbb{S}_1,\mathbb{S}_2, \mathbb{S}_4$ correspond to the horizontal internal edges $e$ and order $\nu_1>\mu_1$. 
The terms $\mathbb{S}_3,\mathbb{S}_5, \mathbb{S}_6$ correspond to the vertical internal edges $e$ and order $\mu_1>\nu_1$. We have 
\begin{multline*}
\mathbb{S}_1+\mathbb{S}_2-\mathbb{S}_3+ \mathbb{S}_4-\mathbb{S}_5-\mathbb{S}_6=\frac{q^{\mu_1+6\nu_1+4}}{(q)_5}-\frac{q^{2\mu_1+5\nu_1+2}}{(q)_1(q)_4}+\frac{q^{3\mu_1+4\nu_1+1}}{(q)_2(q)_3}- \\ -\frac{q^{4\mu_1+3\nu_1+1}}{(q)_3(q)_2}+\frac{q^{5\mu_1+2\nu_1+2}}{(q)_4(q)_1}-\frac{q^{6\mu_1+\nu_1+4}}{(q)_5}=\frac{q^{\mu_1+6\nu_1+4}}{(q)_5}\prod_{i=1}^5(1-q^{\mu_1-\nu_1-3+i})
\end{multline*}
So in other words we get zero for $-2\leq \nu_1-\mu_1\leq 2$ as in Lemma \ref{lem:wallsribbon}.
\end{Example}

\textit{Case 3.}	Union of ribbons containing boxes $(i,H)$ and $(H',j)$ is a connected skew Young diagram $\alpha-\beta$ but not a ribbon. Informally it means that $\alpha-\beta$ has width 2 in the middle. In this case there are two ways to decompose $\alpha-\beta$ into two ribbons. Corresponding two terms are equal to each other, and one goes to $\mathbb{S}^{(H,H')}_{\mathfrak{o}}(q)$ and other to $\mathbb{S}^{(H,H')}_{\mathfrak{o}'}(q)$.
\end{proof} 
\begin{Remark}
	The calculation in Case 3 is essentially the last step in the proof of \cite[Theorem 2.1]{Makhlin} (see our Proposition \ref{pr:Makh}).
\end{Remark}

Tending $H,H' \rightarrow \infty$ we get from Proposition \ref{pr:oo'} that the function $\mathbb{S}_{\mathfrak{o}}(q)$ does not depend on the order $\mathfrak{o}$. For the actual order of $\nu_i,\mu_j$ this function coincides with $\chi_{\mu,\nu,\lambda}^{n,m}(q)$ and for the order \eqref{eq:nu>mu} this function coincides with right side of \eqref{eq:chi:bos'}. Hence we proved Theorem \ref{th:2'} for any pair of partitions $\nu, \mu$, with $l(\nu)\leq n$, $l(\mu)\leq m$.

\section{Algebras, representations and resolutions} \label{sec:alg}

\subsection{\!\!} For the reference of quantum toroidal algebra $U_{\vec{q}}(\ddot{\mathfrak{gl}}_1)$ one can use \cite[Sec. 2]{Feigin_Miwa:2015} or \cite{Tsymbaluk}, but our notation slightly differs from the loc. cit.\footnote{Currently there is no standard convention to the notation, even for the algebra itself other names $\mathcal{E}$, $\mathcal{E}_1$, $\mathbf{SH}$, $\ddot{U}_{q_1,q_2,q_3}(\mathfrak{gl}_1)$ are also used in the literature.}

Fix complex numbers $\epsilon_i$, where $i=1,2,3$  viewed as mod 3 residues. We assume that $\epsilon_1+\epsilon_2+\epsilon_3=0$. Denote $q_i=e^{\epsilon_i}$, $\vec{q}=(q_1,q_2,q_3)$.
We assume further that $q_1, q_2, q_3$ are generic, i.e.,  for integers $l,m,n\in\mathbb{Z}$,  
$q_1^lq_2^mq_3^n=1$ holds only if $l=m=n$. 
We set
\begin{align*}
 g(z,w)=\prod_{i=1}^3 (z-q_iw),\quad
\kappa_r=\prod_{i=1}^3(q_i^{r/2}-q_i^{-r/2})=\sum_{i=1}^3(q_i^{r}-q_i^{-r}), \quad \delta(z)=\sum_{m \in \mathbb{Z}}z^m.
\end{align*}

The algebra $U_{\vec{q}}(\ddot{\mathfrak{gl}}_1)$ is generated by $E_m, F_m, H_{r}$ where $m \in \mathbb{Z}$, $r \in \mathbb{Z}\setminus 0$ and invertible central elements $C, C^\perp$. In order to write down the defining relations we form the currents (generating functions of operators)
\begin{align*}
E(z)=\sum_{m \in \mathbb{Z}}  E_mz^{-m},\;\; F(z)=\sum_{m \in \mathbb{Z}}  F_mz^{-m},\;\;  K^\pm(z)=(C^\perp)^{\pm 1} \exp\left(\sum_{r>0} \mp\frac{\kappa_r}{ r} H_{\pm r} z^{\mp r} \right).
\end{align*}
The relations have form
\begin{align*}
&g(z,w)E(z)E(w)+g(w,z)E(w)E(z)=0, \qquad\quad g(w,z)F(z)F(w)+g(z,w)F(w)F(z)=0,\\
&K^\pm(z)K^\pm(w) = K^\pm(w)K^\pm (z), 
\quad \qquad
\frac{g(C^{-1}z,w)}{g(C z,w)}K^-(z)K^+ (w) 
=
\frac{g(w,C^{-1}z)}{g(w,C z)}K^+(w)K^-(z),
\\
&g(z,w)K^\pm(C^{(-1\mp1)/2}z)E(w)
+g(w,z)E(w)K^\pm(C^{(-1\mp1) /2}z)=0,
\\
&g(w,z)K^\pm(C^{(-1\pm1)/2}z)
F(w)+g(z,w)F(w)K^\pm(C^{(-1\pm1)/2}z)=0\,,
\\
&[E(z),F(w)]=\frac{1}{\kappa_1}
(\delta\bigl(\frac{Cw}{z}\bigr)K^+(w)
-\delta\bigl(\frac{Cz}{w}\bigr)K^-(z)),\\
&\mathop{\mathrm{Sym}}_{z_1,z_2,z_3}z_2z_3^{-1}
[E(z_1),[E(z_2),E(z_3)]]=0,\qquad\quad \mathop{\mathrm{Sym}}_{z_1,z_2,z_3}z_2z_3^{-1}
[F(z_1),[F(z_2),F(z_3)]]=0.
\end{align*}

There exists an action of the group $\widetilde{SL}(2,\mathbb{Z})$ on the toroidal algebra $U_{\vec{q}}(\ddot{\mathfrak{gl}}_1)$ by automorphisms, see \cite[Sec. 6.5]{Burban}. We denote by $E_m^\perp, F_m^\perp, H_r^\perp$ images of generators $E_m, F_m, H_r$  after   rotation of the lattice clockwise by 90 degrees. Under this rotation $C$ goes to $C^\perp$.

Denote by $d$ the operator 
\[
[d, E_m]=-mE_m,\quad [d,F_m]=-mF_m,\quad 
[d, H_{r}] =-r H_{r},\quad [d,C]=[d,C^\perp]=0.
\]
This operator introduces the grading on the algebra $U_{\vec{q}}(\ddot{\mathfrak{gl}}_1)$. Sometimes it is convenient to consider $d$ as an additional generator of $U_{\vec{q}}(\ddot{\mathfrak{gl}}_1)$. Let $V$ be a representation of  $U_{\vec{q}}(\ddot{\mathfrak{gl}}_1)$ such that one can define an action of $d$ on the space $V$ with finite dimensional eigenspaces. By the character $\chi(V)$ we denote the trace of the operator $D=q^{d}$ where $q$ is a formal variable.

The algebra $U_{\vec{q}}(\ddot{\mathfrak{gl}}_1)$  has the following formal coproduct\footnote{Note that our $E_m, F_m, H_r$ are called $e_m^\perp, f_m^\perp, h_r^\perp$ in \cite{Feigin_Miwa:2015} (up to rescaling of $h_r$).} 
\begin{equation}
\begin{aligned}
&\Delta(H_r)=H_r\otimes 1+C^{-r}\otimes H_r,\quad 
\Delta(H_{-r})=H_{-r}\otimes C^{r}+1\otimes H_{-r}, \quad r>0 
\\
&\Delta(E(z))=E\left(C_2^{-1}z\right)\otimes K^+\left(C_2^{-1}z\right)+ 1\otimes E\left(z\right),\\
&\Delta(F(z))=F\left(z\right)\otimes 1 + K^-\left(C_1^{-1}z\right)\otimes F(C_1^{-1}z),\\
&\Delta(X)=X\otimes X,\;\; \text{for $X= C, C^\perp, D$},
\end{aligned} \label{eq:coprod}
\end{equation} 
where $C_1 =C\otimes 1$, $C_2 =1\otimes C$.

In all representations of $U_{\vec{q}}(\ddot{\mathfrak{gl}}_1)$ considered in this paper we have $C^\perp=1$.

In the paper \cite{Feigin_Miwa:2011} authors defined the MacMahon modules of the algebra $U_{\vec{q}}(\ddot{\mathfrak{gl}}_1)$. The MacMahon modules depend on three partitions $\mu,\nu,\lambda$ and two parameters $v, c \in \mathbb{C}$ (the central element $C$ acts on these modules as $c\operatorname{Id}$). These modules are denoted by $\mathcal{M}_{\mu,\nu,\lambda}(v,c)$. The module $\mathcal{M}_{\mu,\nu,\lambda}(v,c)$ has the basis $|a\rangle$, where $a$ is a plane partition which satisfies condition \eqref{eq:asymp}. The action of $d$ on $\mathcal{M}_{\mu,\nu,\lambda}(v,c)$ is defined by $d|a\rangle=|a||a\rangle$. Therefore the character $\chi(\mathcal{M}_{\mu,\nu,\lambda}(v,c))$ is equal to the generating function of plane partitions satisfying \eqref{eq:asymp}.

The modules  $\mathcal{M}_{\mu,\nu,\lambda}(v,c)$ were originally defined by the explicit formulas for the action of ``rotated'' generators $E_m^\perp, F_m^\perp, H_r^\perp$ in the basis labeled by plane partitions. For example, the action of $K^{\perp, \pm}(z)$ have the form 
\begin{equation}\label{eq:Kpm:Mac}
K^{\perp, \pm}(z)|a\rangle=c\frac{1-c^{-2}v/z}{1-v/z}\prod_{(i,j,k)\in a}\psi_{i,j,k}(v/z) |a\rangle
\end{equation}
where 
\[
\psi_{i,j,k}(v/z)=\frac{(1-q_1^{i-1}q_2^jq_3^kv/z)(1-q_1^iq_2^{j-1}q_3^kv/z)(1-q_1^iq_2^jq_3^{k-1}v/z)}{(1-q_1^{i+1}q_2^jq_3^kv/z)(1-q_1^iq_2^{j+1}q_3^kv/z)(1-q_1^iq_2^jq_3^{k+1}v/z)}.
\]
Notation $(i,j,k)\in a$ means that $(i,j,k)$ belongs to the corresponding 3d Young diagram, see Fig. \ref{fig:planepart}. It is easy to see that the product in the right side of \eqref{eq:Kpm:Mac}  becomes finite after cancellation of common factors, see \cite[eq. (3.27)]{Feigin_Miwa:2011}. The highest weight of $\mathcal{M}_{\mu,\nu,\lambda}(v,c)$ is given by the formula \eqref{eq:Kpm:Mac} applied for ``minimal'' plane partition $a$ satisfying conditions~\eqref{eq:asymp}. Note also that here we rather follow \cite{Feigin_Miwa:2015} then \cite{Feigin_Miwa:2011} in notations, for example, we include prefactor $c$ in the rational function \eqref{eq:Kpm:Mac}.


For generic values $c, v, q_1, q_2, q_3$ the module $\mathcal{M}_{\mu,\nu,\lambda}(v,c)$ is irreducible. But for $c=q_1^{n/2}q_2^{m/2}$ (and generic $v, q_1, q_2, q_3$) this module has one singular vector. The quotient by the submodule generated by this vector is irreducible. This quotient is denoted by $\mathcal{N}^{n,m}_{\mu,\nu,\lambda}(v)$ and has the basis $|a\rangle$ where $a$ is a plane partition, satisfying both conditions \eqref{eq:asymp} and \eqref{eq:pit}. Recall that partitions $\lambda, \mu, \nu$ satisfy $l(\nu)\leq n$, $l(\mu)\leq m$, and $\lambda_{n+1} < m+1$.

Therefore the character $\chi(\mathcal{N}^{n,m}_{\mu,\nu,\lambda}(v))$ is equal to the generating function $\chi^{n,m}_{\mu,\nu,\lambda}(q)$ defined in the introduction. This is the representation theoretic interpretation of the left side of \eqref{eq:chi:LGV}, \eqref{eq:chi:bos}, \eqref{eq:chi:bos'}. Now we will discuss the representation theoretic interpretation of the right sides.

\subsection{\!\!}  It is difficult to write down the explicit action of the generators $E_n,F_n,H_m$ in the modules $\mathcal{N}^{n,m}_{\mu,\nu,\lambda}(v)$. Now we recall a construction of another class of modules, namely, the Fock modules and intertwining operators between them, which are called screening operators. Then we sketch a construction of MacMahon modules $\mathcal{N}^{n,m}_{\mu,\nu,\lambda}(v)$ in these terms.

The name of Fock modules over $U_{\vec{q}}(\ddot{\mathfrak{gl}}_1)$ comes from the fact that the representation space is identified with the Fock module over some Heisenberg algebra. In these representations the currents $E(z), F(z), K^\pm(z)$ are given in terms of the Heisenberg algebra (as combination of vertex operators).

We start from the basic Fock modules $\calF_u^{(i)}$, where $u=e^p$, $p \in \mathbb{C}$ and $i=1,2,3$. The underlying vector spaces of these representations are  modules over a Heisenberg algebra with generators $a_n$, $n\in \mathbb{Z}$ and relations
\begin{equation}
[a_r,a_s]=r\frac{(q_i^{r/2}-q_i^{-r/2})^3}{-\kappa_{r}}\delta_{r+s,0}. \label{eq:ar}
\end{equation}
The space $\calF_u^{(i)}$ is generated by the highest weight vector $v_u^{(i)}$ such that 
\[a_rv_u^{(i)}=0 \text{ for }r>0;\qquad a_0v_u^{(i)}=- \frac{ \epsilon_i^2p}{\epsilon_1\epsilon_2\epsilon_3}v_u^{(i)}.\] 
Now we define representation $\rho_u^{(i)}$ of $U_{\vec{q}}(\ddot{\mathfrak{gl}}_1)$ in the space $\calF_u^{(i)}$ by the formulas
\begin{align}
\rho_u^{(i)}(E(z))&=\frac{u(1-q_i)}{\kappa_1}\exp\left(\sum_{r=1}^\infty\frac{q_i^{-r/2}\kappa_r}{r(q_i^{r/2}-q_i^{-r/2})^2}a_{-r}z^r\right)\exp\left(\sum_{r=1}^\infty\frac{\kappa_r}{r(q_i^{r/2}-q_i^{-r/2})^2}a_{r}z^{-r}\right),\notag \\
\rho_u^{(i)}(F(z))&=\frac{u^{-1}(1-q_i^{-1})}{\kappa_1}\exp\left(\sum_{r=1}^\infty\frac{-\kappa_r}{r(q_i^{r/2}-q_i^{-r/2})^2}a_{-r}z^r\right)\exp\left(\sum_{r=1}^\infty\frac{-q_i^{r/2}\kappa_r}{r(q_i^{r/2}-q_i^{-r/2})^2}a_{r}z^{-r}\right),\notag  \\
\rho_u^{(i)}(H_r)&=\frac{a_r}{q_i^{r/2}-q_i^{-r/2}}, \qquad \rho_u^{(i)}(C^\perp)=1,\qquad \rho_u^{(i)}(C)=q_i^{1/2}, \notag  \\
\rho_u^{(i)}(d)v_u^{(i)}&=\Delta_i(p) v_u^{(i)},\text{ where } \Delta_i(p)=-\frac{(p+\epsilon_i)^3-p^3} {6\epsilon_1\epsilon_2\epsilon_3}. \label{eq:Deltap}
\end{align}

Note that generally speaking the operators $a_0$ and $d$ can act on the highest weight vector $v_u^{(i)}$ by any numbers. Our choice is convenient for the formulas below, for example, the screening operators will commute with $d$ due to our choice.

We formally introduce operators $\widehat{Q}$ by the relation $[a_n,\widehat{Q}]=\frac{-\epsilon_i^3}{\epsilon_1\epsilon_2\epsilon_3}\delta_{n,0}$. This operator does not act on the spaces $\calF_u^{(i)}$, but for $x\in \mathbb{C}$ we  define operator  $e^{x\widehat{Q}}\colon \calF^{(i)}_{u}\rightarrow \calF^{(i)}_{q_i^{x}u}$ such that 
\[[a_n,e^{x\widehat{Q}}]=0 \text{ for }n \neq 0,\qquad e^{x\widehat{Q}}v^{(i)}_{u}=v^{(i)}_{q_i^xu}.\]

The $U_{\vec{q}}(\ddot{\mathfrak{gl}}_1)$ modules $\calF_{u}^{(i)}$ are irreducible. In terms of the rotated generators, their highest weight has the form
\begin{equation}\label{eq:hwFock}
K^{\perp, \pm}(z)v_u^{(i)}=q_i^{1/2}\frac{1-q_i^{-1}u/z}{1-u/z} v_u^{(i)}
\end{equation}
In particular, the highest weight of the Fock module $\calF_{u}^{(1)}$ coincides with the highest weight of 
the Macmahon module $\mathcal{M}_{\varnothing,\{\nu_1\},\{\lambda_1\}}(v,c)$ for $c=q_1^{1/2}$ and $u=vq_2^{\lambda_1}q_3^{\nu_1}$ (see \eqref{eq:Kpm:Mac}). Therefore the irreducible quotient $\mathcal{N}^{1,0}_{\varnothing,\{\nu_1\},\{\lambda_1\}}(v)$ is isomorphic to the Fock module $\calF_{u}^{(1)}$. Similarly the MacMahon module $\mathcal{N}^{0,1}_{\{\mu_1\},\varnothing,\{\lambda_1\}^t}(v)$ is isomorphic to the Fock module $\calF_u^{(2)}$, where $u=vq_1^{\lambda_1}q_3^{\mu_1}$.

More generally, we note that the highest weight of the module $\mathcal{N}^{n,m}_{\mu,\nu,\lambda}(v)$ is given by the rational function, see \cite[eq. (3.27)]{Feigin_Miwa:2011}. Numbers of zeros and poles in this rational function are equal, positions of zeros and poles of this rational function differs by the factors of the form $q_1^iq_2^jq_3^k$, $i,j,k \in \mathbb{Z}$, therefore it can can be decomposed as the product of several factors of the type \eqref{eq:hwFock}.  Therefore $\mathcal{N}^{n,m}_{\mu,\nu,\lambda}(v)$ is isomorphic to a subquotient of a tensor product of Fock module
\begin{equation*}
\calF^{(i_1)}_{u_1}\otimes \calF^{(i_2)}_{u_2}\otimes \cdots \otimes \calF^{(i_k)}_{u_k}.
\end{equation*} 
Here we used the fact that the action of $K^{\perp, \pm}(z)$ on the tensor product of highest vectors is given by the product of rational functions \eqref{eq:hwFock} by \cite[Lemma A.9]{Feigin_Miwa:2016} (or conjugacy of coproducts $\Delta$ and $\Delta^\perp$ \cite[Lemma A.13]{Feigin_Miwa:2016}). Clearly such tensor product is also a Fock representation of the sum of $k$ copies of the  Heisenberg algebras or, in other words, Heisenberg algebra with generators $a_{r,i}$, $r \in \mathbb{Z}$, $1 \le i \le k$.  Now we will describe the image of  $U_{\vec{q}}(\ddot{\mathfrak{gl}}_1)$ in these representations.

\subsection{\!\!} First, we consider the tensor product $\calF^{(i)}_{u_1}\otimes \calF^{(i)}_{u_2}$. This tensor product was essentially elaborated in \cite{FHSSY:2010} which we follow.  
We introduce the Heisenberg generators $h_n$, $n \ne 0$ acting on $\calF^{(i)}_{u_1}\otimes \calF^{(i)}_{u_2}$ 
\begin{align}
h_{-r}=q_i^{-r} (a_{-r}\otimes 1)-q_i^{-r/2} (1\otimes a_{-r}), \quad 
h_r=q_i^{r/2} (a_r\otimes 1)-q_i^{r} (1\otimes a_r), \quad r>0.
\end{align}
For any $n \in \mathbb{Z}\setminus\{0\}$ they satisfy $[h_n,\Delta(H_m)]=0$ for all $m \in \mathbb{Z} \setminus\{0\}$ and, clearly, $h_n$ is a unique (up to a normalization) linear combination of $a_n\otimes 1$ and $1\otimes a_n$  with such property. In our normalization we have 
\[[h_r,h_{s}]=r\frac{(q_i^{r}-q_i^{-r})(q_i^{r/2}-q_i^{-r/2})^2}{-\kappa_{r}}\delta_{r+s,0}.\]

Denote $\widehat{Q}_1=\widehat{Q}\otimes 1$, $\widehat{Q}_2=1\otimes \widehat{Q}$, $\widehat{Q}_{12}=\widehat{Q}_1-\widehat{Q}_2$,  $\widehat{a}_{12}=a_0\otimes 1-1\otimes a_0$, $u_1=e^{p_1}$, $u_2=e^{p_2}$. Following \cite{SKAO:1995}, \cite{FeiginFrenkel} we introduce two screening currents
\begin{equation}\label{eq:ScrCur+-}
\begin{aligned}
S_+^{ii}(z)&=e^{\frac{\epsilon_{i+1}}{\epsilon_i}\widehat{Q}_{12}}z^{\frac{\epsilon_{i+1}}{\epsilon_i}\widehat{a}_{12}+\frac{\epsilon_i}{\epsilon_{i-1}}}\exp\left(\sum_{r=1}^\infty\frac{-(q_{i+1}^{r/2}{-}q_{i+1}^{-r/2})}{r(q_i^{r/2}{-}q_i^{-r/2})}h_{-r}z^r\right)\exp\left(\sum_{r=1}^\infty\frac{(q_{i+1}^{r/2}{-}q_{i+1}^{-r/2})}{r(q_i^{r/2}{-}q_{i}^{-r/2})} h_{r}z^{-r}\right)\\
S_-^{ii}(z)&=e^{\frac{\epsilon_{i-1}}{\epsilon_i}\widehat{Q}_{12}}z^{\frac{\epsilon_{i-1}}{\epsilon_i}\widehat{a}_{12}+\frac{\epsilon_i}{\epsilon_{i+1}}}\exp\left(\sum_{r=1}^\infty\frac{-(q_{i-1}^{r/2}{-}q_{i-1}^{-r/2})}{r(q_i^{r/2}{-}q_i^{-r/2})}h_{-r}z^r\right)\exp\left(\sum_{r=1}^\infty\frac{(q_{i-1}^{r/2}{-}q_{i-1}^{-r/2})}{r(q_i^{r/2}{-}q_{i}^{-r/2})} h_{r}z^{-r}\right).
\end{aligned}
\end{equation}
\begin{Lemma}\label{lem:S+-}
The following  screening operators 
\begin{equation}\label{eq:ScrOpe+-}
S_+^{ii}=\oint S_+^{ii}(z)\mathrm{d}z,\quad S_-^{ii}=\oint S_-^{ii}(z)\mathrm{d}z.
\end{equation}
commute with action of $U_{\vec{q}}(\ddot{\mathfrak{gl}}_1)$ (including the operator $d$).
\end{Lemma}
Here $\oint$ is an integral over small contour around 0 (residue at 0). This integral is defined if the corresponding power of $z$ is integer. In other words the operator $ S_+^{ii}$ acts on the tensor product $\calF^{(i)}_{u_1}\otimes \calF^{(i)}_{u_2}$ if and only if $\frac{p_2-p_1+\epsilon_i}{\epsilon_{i-1}}$ (eigenvalue of $\frac{\epsilon_{i+1}}{\epsilon_i}\widehat{a}_{12}+\frac{\epsilon_i}{\epsilon_{i-1}}$) is integer and similarly for $S_-^{ii}$. The lemma  follows from a direct computation, which is given for example in \cite{SKAO:1995} or in \cite[Sec. 5]{FeiginFrenkel} in the equivalent setting of deformed Virasoro (and, more generally, deformed $W$ algebras).

Note that the screening currents formally commute
\[S_+^{ii}(z)S_-^{ii}(w)=\frac{1}{(z-q_i^{-1/2}w)(z-q_i^{1/2}w)}:\!\!S_+^{ii}(z)S_-^{ii}(w)\!\!:=S_-^{ii}(w)S_+^{ii}(z). \]
Here and below, $:\!\! \cdots\!\!:$ is a standard Heisenberg normal ordering, in which $a_{r}$ appears to the right of $a_{-r}$, $r>0$, and $a_0$ appears to the right of $Q$.

We denote by $\mathrm{W}_{\vec{q}}(\dot{\mathfrak{gl}}_2)$ the quotient of $U_{\vec{q}}(\ddot{\mathfrak{gl}}_1)$  by the two-sided ideal generated by operators, which act as 0 on any tensor product $\calF^{(i)}_{u_1}\otimes \calF^{(i)}_{u_2}$. It was proven in \cite[Sec. 2.4]{FHSSY:2010} that this $W$-algebra is just the product of a Heisenberg algebra and a deformed Virasoro algebra introduced in \cite{SKAO:1995}, in other words product of $\mathrm{W}_{\vec{q}}(\dot{\mathfrak{gl}}_1)$ and $\mathrm{W}_{\vec{q}}(\dot{\mathfrak{sl}}_2)$, as expected.

Now on can use the results on the representation theory of deformed Virasoro algebra for the study of representation $\calF^{(i)}_{u_1}\otimes \calF^{(i)}_{u_2}$.\footnote{instead of this one can use \cite{Feigin_Miwa:2010} but the coproduct in loc. cit. is different from \eqref{eq:coprod} used in our paper.} For generic $u_1,u_2$ this module is irreducible. But if $u_2=u_1q_{i+1}^sq_{i-1}^t$ where $s,t \in \mathbb{Z}$ and $st>0$ then it is not so. If $s,t <0$, then $\calF^{(i)}_{u_1}\otimes \calF^{(i)}_{u_2}$ has a singular vector and this singular vector can be obtained by the action of the power of screening operator $S_+^{ii}$ or $S_-^{ii}$ \cite[Sec. 5]{SKAO:1995}.  The quotient of $\calF^{(i)}_{u_1}\otimes \calF^{(i)}_{u_2}$ by the submodule generated by this singular vector is irreducible.\footnote{This follows from the irreducibility of the quotient in conformal limit $q_1,q_2,q_3 \rightarrow 1$, where deformed Virasoro algebra becomes Virasoro algebra and one can use e.g. \cite[Theorem 1.9(ii)]{Feigin Fuchs 1990}} The highest weight of this irreducible quotient coincides with the highest weight of  $\calF^{(i)}_{u_1}\otimes \calF^{(i)}_{u_2}$ and actually coincides with the highest weight of $\mathcal{N}^{2,0}_{\varnothing,\{\nu_1,\nu_2\},\{\lambda_1,\lambda_2\}}(v)$. Since the MacMahon module  $\mathcal{N}^{2,0}_{\varnothing,\{\nu_1,\nu_2\},\{\lambda_1,\lambda_2\}}(v)$ is also irreducible we proved that it is isomorphic to the quotient of  $\calF^{(i)}_{u_1}\otimes \calF^{(i)}_{u_2}$. This can be written as a short exact sequence.

The simplest example of such complex is given if $s$ or $t$ equal to 1 
\begin{equation*}
\begin{aligned}
0 \rightarrow \calF^{(1)}_{vq_2^{\lambda_1}q_3^{\nu_1-1}}\otimes \calF^{(1)}_{vq_{2}^{\lambda_1-s}q_3^{\nu_1}} \xrightarrow{S_-^{11}} \calF^{(1)}_{vq_2^{\lambda_1}q_3^{\nu_1}}\otimes \calF^{(1)}_{vq_{2}^{\lambda_1-s}q_3^{\nu_1-1}} \rightarrow \mathcal{N}^{2,0}_{\varnothing,\{\nu_1,\nu_1\},\{\lambda_1,\lambda_1-s+1\}}(v) \rightarrow 0\\
0 \rightarrow \calF^{(1)}_{vq_2^{\lambda_1-1}q_3^{\nu_1}}\otimes \calF^{(1)}_{vq_2^{\lambda_1}q_{3}^{\nu_1-s}} \xrightarrow{S_+^{11}} \calF^{(1)}_{vq_2^{\lambda_1}q_3^{\nu_1}}\otimes \calF^{(1)}_{vq_2^{\lambda_1-1}q_{3}^{\nu_1-s}} \rightarrow \mathcal{N}^{2,0}_{\varnothing,\{\nu_1,\nu_1-s+1\},\{\lambda_1,\lambda_1\}}(v) \rightarrow 0
\end{aligned}.
\end{equation*} 
For a generic module $\mathcal{N}^{2,0}_{\varnothing,\{\nu_1,\nu_2\},\{\lambda_1,\lambda_2\}}(v)$ the corresponding short exact sequences have the form 
\begin{equation}\label{eq:resgl2}
\begin{aligned}
0 \rightarrow \calF^{(1)}_{vq_2^{\lambda_1}q_3^{\nu_2-1}}\otimes \calF^{(1)}_{vq_{2}^{\lambda_2-1}q_3^{\nu_1}} \xrightarrow{(S_-^{11})^{\nu_1-\nu_2+1}} \calF^{(1)}_{vq_2^{\lambda_1}q_3^{\nu_1}}\otimes \calF^{(1)}_{vq_{2}^{\lambda_2-1}q_3^{\nu_2-1}} \rightarrow \mathcal{N}^{2,0}_{\varnothing,\{\nu_1,\nu_2\},\{\lambda_1,\lambda_2\}}(v) \rightarrow 0\\
0 \rightarrow \calF^{(1)}_{vq_2^{\lambda_2-1}q_3^{\nu_1}}\otimes \calF^{(1)}_{vq_{2}^{\lambda_1}q_3^{\nu_2-1}} \xrightarrow{(S_+^{11})^{\lambda_1-\lambda_2+1}} \calF^{(1)}_{vq_2^{\lambda_1}q_3^{\nu_1}}\otimes \calF^{(1)}_{vq_{2}^{\lambda_2-1}q_3^{\nu_2-1}} \rightarrow \mathcal{N}^{2,0}_{\varnothing,\{\nu_1,\nu_2\},\{\lambda_1,\lambda_2\}}(v) \rightarrow 0
\end{aligned},
\end{equation}
where operators $(S_{\pm}^{11})^{r}$ should be considered as $r$-fold integrals over the appropriate cycle with the appropriate additional (Lukyanov) factor (see \cite{Jimbo}). Now we compute the Euler characteristic of \eqref{eq:resgl2}.  Using $\chi(\calF^{(1)}_{u_1}\otimes \calF^{(1)}_{u_2})={q^{\Delta_1(p_1)+\Delta_1(p_2)}}/{(q)_\infty^2}$, where $\Delta_i(p)$ is defined in \eqref{eq:Deltap} we have 
\begin{equation}\label{eq:NF1F1}
\chi\Bigl(\mathcal{N}^{2,0}_{\varnothing,\{\nu_1,\nu_2\},\{\lambda_1,\lambda_2\}}(v)\Bigr)=q^\Delta\,\frac{q^{-(\lambda_1+1)(\nu_1+1)-\lambda_2\nu_2}-q^{-(\lambda_1+1)\nu_1-\lambda_2(\nu_2+1)}}{(q)_\infty^2},
\end{equation}
where \[\Delta=\Delta_1(p+\lambda_1\epsilon_2+\nu_1\epsilon_3)+\Delta_1(p+(\lambda_2-1)\epsilon_2+(\nu_2-1)\epsilon_3)+(\lambda_1+1)(\nu_1+1)+\lambda_2\nu_2,\]
and $e^p=v$.
 Up to the factor $q^{\ldots}$ the formula \eqref{eq:NF1F1} coincides with \eqref{eq:chi:bos'} (or with its special case~\eqref{eq:Wn:char}).


In a similar manner one can construct resolutions of 
$\mathcal{N}_{\{\mu_1,\mu_2\},\varnothing,\{\lambda_1,\lambda_2\}^t}^{0,2}(v)$ in terms of $\calF_{u_1}^{(2)}\otimes \calF_{u_2}^{(2)}$.

Below we will discuss the algebra of screening operators which commute with the image of algebra $U_{\vec{q}}(\ddot{\mathfrak{gl}}_1)$ in the representation $\calF^{(1)}_{u_1}\otimes\ldots \otimes \calF^{(1)}_{u_n}$. This system of screening operators coincides with the one studied in \cite{FeiginFrenkel}, the image of $U_{\vec{q}}(\ddot{\mathfrak{gl}}_1)$ commutes with them and is denoted by  $\mathrm{W}_{\vec{q}}(\dot{\mathfrak{gl}}_n)$. We conjecture that one can construct resolutions of $\mathcal{N}_{\varnothing,\nu,\lambda}^{n,0}(v)$ in terms of the modules $\calF^{(1)}_{u_1}\otimes\ldots \otimes\calF^{(1)}_{u_n}$. See Section \ref{ssec:Fnm} for more details.

\subsection{\!\!} Second, we consider the tensor product $\calF^{(1)}_{u_1}\otimes \calF^{(2)}_{u_2}$. We introduce the Heisenberg generators $h_n$ acting on $\calF^{(1)}_{u_1}\otimes \calF^{(2)}_{u_2}$ as follows
\begin{equation}
\begin{aligned}
h_{-r}&=\frac{q_1^{-r}(q_2^{r/2}-q_2^{-r/2})}{q_1^{r/2}-q_1^{-r/2}}(a_{-r}\otimes 1)-\frac{q_1^{-r/2}(q_1^{r/2}-q_1^{-r/2})}{q_2^{r/2}-q_2^{-r/2}}(1\otimes a_{-r}), 
\\
h_r&=\frac{q_1^{r/2}(q_2^{r/2}-q_2^{-r/2})}{q_1^{r/2}-q_1^{-r/2}}(a_r\otimes 1)-\frac{q_3^{-r/2}(q_1^{r/2}-q_1^{-r/2})}{q_2^{r/2}-q_2^{-r/2}}(1\otimes a_r), 
\end{aligned}
\quad r>0.
\end{equation}
For any $n \in \mathbb{Z}\setminus\{0\}$ they satisfy $[h_n,\Delta(H_m)]=0$ for all $m \in \mathbb{Z} \setminus\{0\}$ and, clearly, $h_n$ is a unique (up to a normalization) linear combination of $a_n\otimes 1$ and $1\otimes a_n$  with such property. In our normalization we have 
\[[h_r,h_{s}]=r\delta_{r+s,0}.\]
Similarly to the previous case, denote $\widehat{Q}_1=\widehat{Q}\otimes 1$, $\widehat{Q}_2= 1\otimes \widehat{Q}$, $\widehat{a}_1=a_0\otimes 1$, $\widehat{a}_2= 1\otimes a_0$, $u_1=e^{p_1}$, $u_2=e^{p_2}$ and introduce a screening current
\begin{equation}\label{eq:S12}
S^{12}(z)=e^{\frac{\epsilon_2}{\epsilon_1} \widehat{Q}_1-\frac{\epsilon_1}{\epsilon_2} \widehat{Q}_2}z^{\frac{\epsilon_2}{\epsilon_1} \widehat{a}_1-\frac{\epsilon_1}{\epsilon_2} \widehat{a}_2+\frac{\epsilon_2}{\epsilon_3}}\exp\left(\sum_{r=1}^\infty \frac{1}{-r} h_{-r}z^r\right)\exp\left(\sum_{r=1}^\infty \frac{1}{r} h_{r}z^{-r}\right).
\end{equation}
\begin{Lemma}\label{lem:S12}
The following  screening operator 
\begin{equation}\label{eq:ScrOpe12}
S^{12}=\oint S^{12}(z)\mathrm{d}z.
\end{equation}
commutes with action of $U_{\vec{q}}(\ddot{\mathfrak{gl}}_1)$ (including the operator $d$). 
\end{Lemma}
About the meaning of the integration see the  comment after the Lemma \ref{lem:S+-}.  The proof of the lemma is similar to the one of the Lemma \ref{lem:S+-} and based on a standard computation.
\begin{proof}
Clearly the action of $\Delta(K^{\pm}(z))$ commutes with whole current $S^{12}(w)$ since the latter is expressed in terms of $h_n$. Now we show that $\Delta(E(z))$ commutes with $S^{12}$.

Denote 
\[\Lambda_1(z)=E\left(q_2^{-1}z\right)\otimes K^+\left(q_2^{-1}z\right),\qquad \Lambda_2(z)=1\otimes E\left(z\right).\]
Then $\Delta(E(z))=\Lambda_1(z)+\Lambda_2(z)$. We have
	\begin{align*}
\Lambda_1(z)S^{12}(w)&=q_2\frac{1-w/z}{1-q_{2}w/z}:\!\!\Lambda_1(z)S^{12}(w)\!\!:, 
&\quad 
S^{12}(w)\Lambda_1(z)&=\frac{1-z/w}{1-q_2^{-1}z/w}:\!\!\Lambda_1(z)S^{12}(w)\!\!:
\\
\Lambda_2(z)S^{12}(w)&=q_1^{-1}\frac{1-w/z}{1-q_{1}^{-1}w/z}:\!\!\Lambda_2(z)S^{12}(w)\!\!:, 
&\quad 
S^{12}(w)\Lambda_2(z)&=\frac{1-z/w}{1-q_1z/w}:\!\!\Lambda_2(z)S^{12}(w)\!\!:
\end{align*} 
Therefore we have
	\begin{align*}
\left[\Lambda_1(z),S^{12}(w)\right]&=(q_2-1)\delta\left(\frac{q_{2}^{-1}z}{w}\right):\!\!\Lambda_1(z)S^{12}(w)\!\!:=
(q_{2}-1):\!\!\Lambda_1(z)S(q_{2}^{-1}z)\!\!:\delta\left(\frac{q_{2}^{-1}z}{w}\right)\\
\left[\Lambda_2(z),S^{12}(w)\right]&=(q_1^{-1}-1)\delta\left(\frac{q_1z}{w}\right):\!\!\Lambda_2(z)S^{12}(w)\!\!:=
(q_{1}^{-1}-1):\!\!\Lambda_2(z)S^{12}(q_{1}z)\!\!:\delta\left(\frac{q_{1}z}{w}\right).
\end{align*}
Finally we get  \[\left[\Lambda_1(z)+\Lambda_2(z),S^{12}(w)\right]=\mathfrak{D}_{q_{3}}\left(w(q_{2}-1):\!\!\Lambda_1(z)S^{12}(q_{2}^{-1}z)\!\!:\delta\left(\frac{q_{2}^{-1}z}{w}\right)\right),\] where we used
	\begin{equation*}
(1-q_2):\!\!\Lambda_1(z)S^{12}(q_2^{-1}z)\!\!:=q_2(1-q_1):\!\!\Lambda_2(z)S^{12}(q_1z)\!\!: 
\end{equation*}
and notation $[\mathfrak{D}_af](w)=\frac{f(w)-f(wa)}{w}$. Therefore $\Delta(E(z))$ commutes with $S^{12}=\oint S^{12}(z)dz$, the computation for $\Delta(F(z))$ is similar. 
\end{proof}

Note that in this case we have only one screening current contrary to two currents $S_-^{ii}(z), S_+^{ii}(z)$ above in \eqref{eq:ScrCur+-}. Also note that the commutation relations of $S^{12}(z)$ do not depend on $q_1,q_2,q_3$, namely
\[S^{12}(z)S^{12}(w)=(z-w):\!\!S^{12}(z)S^{12}(w)\!\!:=- S^{12}(w)S^{12}(z).\]
In other words, we have a ``fermionic'' screening current. In particular, we have $S^{12}S^{12}=0$. 

We denote by $\mathrm{W}_{\vec{q}}(\dot{\mathfrak{gl}}_{1|1})$ the quotient of $U_{\vec{q}}(\ddot{\mathfrak{gl}}_1)$  by the two-sided ideal generated by operators, which act as 0 on any tensor product $\calF^{(1)}_{u_1}\otimes \calF^{(2)}_{u_2}$. The arguments for such name will be given below.

For generic $u_1,u_2$ the module $\calF^{(1)}_{u_1}\otimes \calF^{(2)}_{u_2}$ is irreducible. But in the resonance case we have  nontrivial intertwining operators between such modules and can construct a complex with cohomology $\mathcal{N}^{1,1}_{\{\mu_1\},\{\nu_1\},\varnothing}(v)$. 

There are two infinite exact sequences  
\begin{equation} \label{eq:resgl11}
 \ldots \xrightarrow{S^{12}}\calF^{(1)}_{vq_2^{-2}q_3^{\nu_1}}\otimes \calF^{(2)}_{vq_1^3q_3^{\mu_1}} \xrightarrow{S^{12}}\calF^{(1)}_{vq_2^{-1}q_3^{\nu_1}}\otimes \calF^{(2)}_{vq_1^2q_3^{\mu_1}} \xrightarrow{S^{12}} \calF^{(1)}_{vq_3^{\nu_1}}\otimes \calF^{(2)}_{vq_1q_3^{\mu_1}} \rightarrow \mathcal{N}^{1,1}_{\{\mu_1\},\{\nu_1\},\varnothing}(v) \rightarrow 0,
\end{equation}
\begin{equation} \label{eq:resgl11'}
0 \rightarrow  \mathcal{N}^{1,1}_{\{\mu_1\},\{\nu_1\},\varnothing}(v) \xrightarrow{S^{12}} \calF^{(1)}_{vq_2q_3^{\nu_1}}\otimes \calF^{(2)}_{vq_3^{\mu_1}}\xrightarrow{S^{12}} \calF^{(1)}_{vq_2^2q_3^{\nu_1}}\otimes \calF^{(2)}_{vq_1^{-1}q_3^{\mu_1}}\xrightarrow{S^{12}} \calF^{(1)}_{vq_2^3q_3^{\nu_1}}\otimes \calF^{(2)}_{vq_1^{-2}q_3^{\mu_1}} \xrightarrow{S^{12}} \ldots  
\end{equation}
Of course, the exactness is nontrivial, the proof will be given elsewhere.

Now we calculate the Euler characteristics of the exact sequence \eqref{eq:resgl11}. Introduce $p$ by $e^p=v$ and $\Delta=\Delta_1(p+\nu_1\epsilon_3)+\Delta_2(p+\epsilon_1+\mu_1\epsilon_3)$. We have an algebraic identity (using the definition \eqref{eq:Deltap})
\[
\Delta_1(p{+}\nu_1\epsilon_3{-}a\epsilon_2)+\Delta_2(p{+}(a{+}1)\epsilon_1{+}\mu_{1}\epsilon_3)-\Delta=\binom{a+1}{2}+(\nu_1-\mu_1)a.
\]
Therefore
\begin{equation} \label{eq:F11=R}
\chi\Bigl(\mathcal{N}^{1,1}_{\{\mu_1\},\{\nu_1\},\varnothing}(v)\Bigr)=
q^\Delta R(\nu_1-\mu_1;q)
\end{equation}
where function $R$ is defined in \eqref{eq:R:defin}. Up to a factor $q^{\ldots}$ this formula coincides with \eqref{eq:chi:bos'} for the special case $n=m=1$.

Similarly, computing the Euler characteristic of the exact sequence  \eqref{eq:resgl11'} one gets  \[\chi\Bigl(\mathcal{N}^{1,1}_{\{\mu_1\},\{\nu_1\},\varnothing}(v)\Bigr)=
q^{\Delta_1(p{+}\nu_1\epsilon_3{+}\epsilon_2)+\Delta_2(p{+}\mu_{1}\epsilon_3)} R(\mu_1-\nu_1;q).\]
Of course, the fact that $R(\nu_1-\mu_1;q)$ and $R(\mu_1-\nu_1;q)$ are proportional is not surprising, it is equivalent to proportionality of $\chi_{{\mu_1},{\nu_1},\varnothing}^{1,1}(q)$ and $\chi_{{\nu_1},{\mu_1},\varnothing}^{1,1}(q)$, which follows from the bijection given by transposition of plane partition $a$. See also example \ref{ex:n=m=1} above.

\subsection{\!\!} \label{ssec:Fnm} Now we want to consider  tensor products
\begin{equation}\label{eq:Fnm}
\calF^{(1)}_{u_1}\otimes\ldots \otimes \calF^{(1)}_{u_n}\otimes \calF^{(2)}_{u_{n+1}} \otimes \ldots \otimes \calF^{(2)}_{u_{n+m}}.
\end{equation}
Similarly to previous discussion we expect that the MacMahon module $\mathcal{N}^{n,m}_{\mu,\nu,\lambda}(v)$ admits a resolution consisting of modules of the type \eqref{eq:Fnm}. For example one can easily see that the central charge of $\mathcal{N}^{n,m}_{\mu,\nu,\lambda}(v)$ and central charge of  \eqref{eq:Fnm} are both equal to $c=q_1^{n/2}q_2^{m/2}$.

Moreover we can consider another ordering in the tensor product \eqref{eq:Fnm}. For generic parameters $u$ any tensor product of $n$ modules $\calF^{(1)}_{u_i}$ and $m$ modules $\calF^{(2)}_{u_j}$ is isomorphic to the product in the ordering \eqref{eq:Fnm}. For each pair of neighbor Fock modules $\calF^{(i_l)}_{u_l}\otimes \calF^{(i_{l+1})}_{u_{l+1}}$ we constructed above the screening operators $\left(S_*^{i_l i_{l+1}}\right)_{l,l{+}1}$, where indices $l, l{+}1$ label Fock modules in which this operator acts and $*=\pm$ if $i_l= i_{l+1}$, while $*$ should be ignored when $i_l\ne i_{l+1}$. For any $l,*$ the operator $\left(S_*^{i_l i_{l+1}}\right)_{l,l{+}1}$ commutes with $U_{\vec{q}}(\ddot{\mathfrak{gl}}_1)$.

For modules \eqref{eq:Fnm} it is convenient to decompose the corresponding screening operators into 3 systems:
\begin{equation}
\begin{aligned}
\mathfrak{S}_{1}&= \left\{\left(S_-^{11}\right)_{i,i+1}| 0< i< n \right\},\\ 
\mathfrak{S}_{2}&= \left\{\left(S_+^{22}\right)_{j,j+1}|n< j < n+m\right\},\\
\mathfrak{S}_3&= \left\{\left(S_+^{11}\right)_{i,i+1}, \left(S^{12}\right)_{n,n+1}, \left(S_-^{22}\right)_{j,j+1}| 0< i <n,\; n< j < n+m\right\}.
\end{aligned}
\end{equation}

We will denote by $\mathrm{W}_{\vec{q}}(\dot{\mathfrak{gl}}_{n|m})$ the quotient of $U_{\vec{q}}(\ddot{\mathfrak{gl}}_1)$  by the two-sided ideal generated by operators, which act as 0 on any tensor product \eqref{eq:Fnm}.

Now we discuss a representation theoretic interpretation of the character formulas \eqref{eq:chi:bos'}, \eqref{eq:chi:bos}, \eqref{eq:chi:LGV}.

$\bullet$ Each term in the sum of right side of \eqref{eq:chi:bos'} has the form of $q^\Delta/(q)_\infty^{n+m}$. This is the character of the Fock module \eqref{eq:Fnm}. Therefore it is natural to expect that right side of~\eqref{eq:chi:bos'} is an Euler characteristic of a resolution, consisting of Fock modules \eqref{eq:Fnm}. The terms in this resolution should be labeled by $(\sigma,\tau,A)\in \Theta$ as in \eqref{eq:chi:bos'}. We will say that this resolution is a materialization of the formula \eqref{eq:chi:bos'}.

This resolution of $\mathrm{W}_{\vec{q}}(\dot{\mathfrak{gl}}_{n|m})$ modules is a generalization of resolutions \eqref{eq:resgl2} and \eqref{eq:resgl11}, \eqref{eq:resgl11'} discussed above. The intertwining operators in this resolution could be constructed using screening operators. The construction of such resolution is unknown (actually we did not give a proof of the existence of \eqref{eq:resgl2} and \eqref{eq:resgl11},\eqref{eq:resgl11'} but these particular cases are rather easy). 

Even in the case $m=0$ which corresponds to $\mathrm{W}_{\vec{q}}(\dot{\mathfrak{gl}}_n)$ we did not find this resolution in the literature. For the construction of intertwining operators in the $\mathrm{W}_{\vec{q}}(\dot{\mathfrak{gl}}_n)$ case see \cite{FJMOP}.

$\bullet$ The algebra $U_{\vec{q}}(\ddot{\mathfrak{gl}}_1)$ acts on the product of first $n$ factors of \eqref{eq:Fnm} through the quotient $\mathrm{W}_{\vec{q}}(\dot{\mathfrak{gl}}_n)$ and acts on the product of last $m$ factors of \eqref{eq:Fnm} through the quotient $\mathrm{W}_{\vec{q}}(\dot{\mathfrak{gl}}_m)$. Therefore the coproduct $\Delta \colon U_{\vec{q}}(\ddot{\mathfrak{gl}}_1)\rightarrow U_{\vec{q}}(\ddot{\mathfrak{gl}}_1) \otimes U_{\vec{q}}(\ddot{\mathfrak{gl}}_1)$ provides an embedding  $\mathrm{W}_{\vec{q}}(\dot{\mathfrak{gl}}_{n|m}) \rightarrow  \mathrm{W}_{\vec{q}}(\dot{\mathfrak{gl}}_n)\otimes \mathrm{W}_{\vec{q}}(\dot{\mathfrak{gl}}_m)$. The latter  algebra has representations of the form $\mathcal{N}^{n,0}_{\varnothing,\nu,\tilde{\nu}}(v_1)\otimes \mathcal{N}^{0,m}_{\mu,\varnothing,\tilde{\mu}}(v_2)$. Since the character of each factor is given by formula \eqref{eq:Wn:char} we have 
\[
\chi(\mathcal{N}^{n,0}_{\varnothing,\nu,\tilde{\nu}}(v_1)\otimes \mathcal{N}^{0,m}_{\mu,\varnothing,\tilde{\mu}}(v_2))=q^\Delta\frac{a_{\nu+\rho}(q^{-\tilde{\nu}-\rho})a_{\mu+\rho}(q^{-\tilde{\mu}-\rho})}{(q)_\infty^{n+m}},
\]
for certain $\Delta$. The right side of the character formula \eqref{eq:chi:bos} is a linear combination of such terms. Therefore it is natural to expect that there exists a resolution of $\mathcal{N}^{n,m}_{\mu,\nu,\lambda}(v)$ consisting of modules $\mathcal{N}^{n,0}_{\varnothing,\nu,\tilde{\nu}}(v_1)\otimes \mathcal{N}^{0,m}_{\mu,\varnothing,\tilde{\mu}}(v_2)$ with the cohomology $\mathcal{N}^{n,m}_{\mu,\nu,\lambda}(v)$. This resolution should be a materialization of the character formula  \eqref{eq:chi:bos}. See also Section \ref{ssec:glnm} below.

$\bullet$ Consider tensor product 
\begin{equation} \label{eq:Fnm:2}
\calF^{(1)}_{u_1}\otimes\ldots \otimes\calF^{(1)}_{u_{n-r}}\otimes\ \calF^{(2)}_{u_{n-r+1}} \otimes \ldots \otimes \calF^{(2)}_{u_{n+m-2r}}\otimes \Bigl( \calF^{(1)}_{u_{n+m-2r+1}}\otimes \calF^{(2)}_{u_{n+m-2r+2}}\otimes\ldots \otimes\calF^{(1)}_{u_{n+m-1}}\otimes \calF^{(2)}_{u_{n+m}}\Bigr).
\end{equation}
As mentioned before this product is isomorphic to \eqref{eq:Fnm} (for generic parameters $u_l$).

The algebra $U_{\vec{q}}(\ddot{\mathfrak{gl}}_1)$ acts on the product \eqref{eq:Fnm:2} through the algebra $\mathrm{W}_{\vec{q}}(\dot{\mathfrak{gl}}_{1})^{\otimes(n+m-2r)}\otimes \mathrm{W}_{\vec{q}}(\dot{\mathfrak{gl}}_{1|1})^{\otimes r}$, where $\mathrm{W}_{\vec{q}}(\dot{\mathfrak{gl}}_{1})^{\otimes n+m-2r}$ is by definition just Heisenberg algebra acting on the first $n+m-2r$ factors of \eqref{eq:Fnm:2}. Therefore $\mathrm{W}_{\vec{q}}(\dot{\mathfrak{gl}}_{n|m})$ is a subalgebra of $\mathrm{W}_{\vec{q}}(\dot{\mathfrak{gl}}_{1})^{\otimes n+m-2r}\otimes \mathrm{W}_{\vec{q}}(\dot{\mathfrak{gl}}_{1|1})^{\otimes r}$. The latter $W$-algebra has  representations 
\begin{equation} \label{eq:RrFm+n-2r}
\calF^{(1)}_{u_1}\otimes\ldots\otimes \calF^{(1)}_{u_{n-r}}\otimes \calF^{(2)}_{u_{n-r+1}} \otimes \ldots \otimes \calF^{(2)}_{u_{n+m-2r}}\otimes \mathcal{N}_{\{m_1\},\{n_1\},\varnothing}^{1,1}(v_1)\otimes \ldots \otimes \mathcal{N}_{\{m_r\},\{n_r\},\varnothing}^{1,1}(v_r).
\end{equation}
Due to \eqref{eq:F11=R} the character of these representations are equal to
$\prod_{i=1}^r R(d_i;q)\cdot {q^\Delta}/{(q_\infty)^{m+n-2r}}$,
where $d_i=n_i-m_i$ and $\Delta$ is specified by parameters $u_iv_i,m_i,n_i$.

If we compute the determinant in right side of \eqref{eq:chi:LGV} we get the linear combination of terms  $\prod_{i=1}^r R(d_i;q)\cdot {q^\Delta}/{(q_\infty)^{m+n-2r}}$ as before. Therefore it is natural to conjecture the existence of the resolution of $\mathcal{N}^{n,m}_{\mu,\nu,\lambda}(v)$ consisting of modules of the type \eqref{eq:RrFm+n-2r}. And this resolution should be a materialization of the  character formula  \eqref{eq:chi:LGV}.

\begin{Remark}\label{rem:Wmnk}
One can consider generic tensor product, 
\[\mathcal{F}^{(1)}_{u_1}\otimes \cdots\otimes \mathcal{F}^{(1)}_{u_n} \otimes \mathcal{F}^{(2)}_{u_{n+1}}\otimes\cdots\otimes \mathcal{F}^{(2)}_{u_{n+m}}\otimes \mathcal{F}^{(3)}_{u_{n+m+1}}\otimes\cdots\otimes	 \mathcal{F}^{(3)}_{u_{n+m+k}}.
\]
The quotient of $U_{\vec{q}}(\ddot{\mathfrak{gl}}_1)$  by the two-sided ideal generated by operators, which act as 0 on any such tensor product is a certain $\mathrm{W}$ algebra depending on three integer parameters $k,n,m$. These algebras (or, rather, their conformal limits) appear recently in \cite{Litvinov}, \cite{Gaiotto} in a different context, the formulas for screening operators for these algebras follows from the Lemmas \ref{lem:S+-}, \ref{lem:S12}. We hope to study these algebras in more details elsewhere. 
\end{Remark}

\subsection{\!\!} 
Now we consider the conformal limit of the previous constructions. We rescale $\epsilon_i\rightarrow
 \hbar\epsilon_i$, $p_l\rightarrow \hbar p_l$ and then send $\hbar$ to zero, i.e., send all parameters $q_i,u_l$ to 1 with a certain speed. Now consider the limit of screening operators, we will see that this limit  is well defined. 

Let $\mathbb{C}^{n+m}$ be a vector space with the scalar product $(\cdot,\cdot)$ given in orthogonal basis $e_l$ by the formula
\[(e_i,e_i)=-\frac{\epsilon_1^2}{\epsilon_2\epsilon_3}, \;\; (e_j,e_j)=-\frac{\epsilon_2^2}{\epsilon_1\epsilon_3},\;\;     1\leq i \leq n,\; n+1 \leq j \leq n+m.\]

Denote by $\bar{a}_{n,l}$, $n\neq 0$ the limit of generators $a_n$ acting on the $l$-th factor in \eqref{eq:Fnm}. As the limit of \eqref{eq:ar} we get
\[[\bar{a}_{r,i},\bar{a}_{s,j}]=-r(e_i,e_j)\delta_{r+s,0}.\]
By $\bar{Q}_l$ we denote is a limit of $\widehat{Q}_l$, and by $\bar{a}_{0,l}$ we denote  
\[\bar{a}_{0,i}=\lim_{\hbar \rightarrow 0} a_{0,i}-i\frac{\epsilon_1^2}{\epsilon_2\epsilon_3},\;\;\bar{a}_{0,j}=\lim_{\hbar \rightarrow 0} a_{0,j}-n\frac{\epsilon_2}{\epsilon_3}-(j-n)\frac{\epsilon_2^2}{\epsilon_1\epsilon_3},\;\;     1\leq i \leq n,\; n+1 \leq j \leq n+m, \]
here such strange shifts are introduced in order to hide factors like $z^{\epsilon_2/\epsilon_3}$ in the definition of $S^{12}$ in \eqref{eq:S12}. Anyway, we get a relation $[\bar{a}_{0,l},\bar{Q}_{l'}]=(e_l,e_{l'})$. As before, the operator $\bar{Q}_{l}$ does not act on any Fock module but the exponent $e^{x\bar{Q}_l}$ acts from one Fock module to another shifting the eigenvalue of $a_{0,l}$.

It is convenient to introduce 
\begin{align} \label{eq:varphi_l}
\varphi_l(z)=\sum_{r \in \mathbb{Z}\setminus 0} \dfrac{\bar{a}_{r,l}}{r} z^{-r}+\bar{a}_{0,l}\log z+\bar{Q}_l, 
\end{align}
Then, the limits of screening currents have the form
\begin{equation*}
\begin{aligned}
\lim_{\hbar \rightarrow 0} (S_{\pm}^{11}(z))_{i,i+1}&=:\!\!\exp\Bigl(\sum\nolimits_{l=1}^{n+m}(\alpha_{\pm,i})_l\varphi_l(z)\Bigr)\!\!:,\;\; 0< i < n,
\\
\lim_{\hbar \rightarrow 0} (S^{12}(z))_{n,n+1}&=:\!\!\exp\Bigl(\sum\nolimits_{l=1}^{n+m}(\alpha_{n})_l\varphi_l(z)\Bigr)\!\!:,\\ 
\lim_{\hbar \rightarrow 0} (S_{\pm}^{22}(z))_{j,j+1}&=:\!\!\exp\Bigl(\sum\nolimits_{l=1}^{n+m}(\alpha_{\pm,j})_l\varphi_l(z)\Bigr)\!\!:, \;\; n < j < n+m.
\end{aligned}
\end{equation*}
We consider $\alpha_{\pm,i}, \alpha_n, \alpha_{\pm,j}$ as the vectors in $\mathbb{C}^{n+m}$ which have the form
\begin{equation*}
\begin{aligned}
\alpha_{+,i}&=\frac{\epsilon_2}{\epsilon_1}e_i-\frac{\epsilon_2}{\epsilon_1}e_{i+1},\quad \alpha_{n}=\frac{\epsilon_2}{\epsilon_1}e_n-\frac{\epsilon_1}{\epsilon_2}e_{n+1},\quad &\alpha_{-,j}&=\frac{\epsilon_1}{\epsilon_2}e_j-\frac{\epsilon_1}{\epsilon_2}e_{j+1}\\
\alpha_{-,i}&=\frac{\epsilon_3}{\epsilon_1}e_i-\frac{\epsilon_3}{\epsilon_1}e_{i+1},\quad &\alpha_{+,j}&=\frac{\epsilon_3}{\epsilon_2}e_j-\frac{\epsilon_3}{\epsilon_2}e_{j+1}.
\end{aligned}
\end{equation*}
Slightly abusing notation we will say that $\alpha \in \mathfrak{S}_{I}$, for $I=1, 2, 3$ if the corresponding screening operator belongs to $\mathfrak{S}_{I}$.

 For any $\beta,\gamma \in \mathbb{C}^{n+m}$ the 
commutation relations of vertex operators have the form 
\begin{multline*}
:\!\!\exp\Bigl(\sum\nolimits_{l=1}^{n+m}\beta_l\varphi_l(z)\Bigr)\!\!:\,\cdot \, :\!\!\exp\Bigl(\sum\nolimits_{l=1}^{n+m}\gamma_l\varphi_l(w)\Bigr)\!\!:\\ =(z-w)^{(\beta,\gamma)} :\!\!\exp\Bigl(\sum\nolimits_{l=1}^{n+m}\beta_l\varphi_l(z)+\gamma_l\varphi_l(w)\Bigr)\!\!:. 
\end{multline*}
In particular, if $(\beta,\gamma) \in 2\mathbb{Z}$ then the corresponding vertex operators formally commute and if $(\beta,\gamma) \in 2\mathbb{Z}+1$ then the corresponding vertex operators formally anticommute. It is easy to see that if two vectors $\alpha, \alpha'$ belong to different systems $\mathfrak{S}$ then the scalar product $(\alpha,\alpha') \in \mathbb{Z}$, i.e., corresponding screening operators formally commute or anticommute. The Gramian matrices for vectors from $\mathfrak{S}_1, \mathfrak{S}_2, \mathfrak{S}_3$ are given below
\[ 
\mathfrak{S}_{1}\colon
\begin{pmatrix}
\frac{-2\epsilon_3}{\epsilon_2}& \frac{\epsilon_3}{\epsilon_2} &  0& \ldots & 0
\\
\frac{\epsilon_3}{\epsilon_2} & \frac{-2\epsilon_3}{\epsilon_2}   &\ddots & \ddots  & \vdots
\\
0 & \ddots  & \ddots & \ddots& 0
\\
\vdots &  \ddots &  \ddots &\frac{-2\epsilon_3}{\epsilon_2} & \frac{\epsilon_3}{\epsilon_2}
\\
0 & \ldots   &0 &  \frac{\epsilon_3}{\epsilon_2}& \vphantom{\biggl(}\frac{-2\epsilon_3}{\epsilon_2}
\end{pmatrix},\quad
\mathfrak{S}_{2}\colon
\begin{pmatrix}
\frac{-2\epsilon_3}{\epsilon_1}& \frac{\epsilon_3}{\epsilon_1} &  0& \ldots & 0
\\
\frac{\epsilon_3}{\epsilon_1} & \frac{-2\epsilon_3}{\epsilon_1}   &\ddots & \ddots  & \vdots
\\
0 & \ddots  & \ddots & \ddots& 0
\\
\vdots &  \ddots &  \ddots &\frac{-2\epsilon_3}{\epsilon_1} & \frac{\epsilon_3}{\epsilon_1}
\\
0 & \ldots   &0 &  \frac{\epsilon_3}{\epsilon_1}& \vphantom{\biggl(}\frac{-2\epsilon_3}{\epsilon_1}
\end{pmatrix}, \]
\[
\mathfrak{S}_{3}\colon
\begin{pmatrix}
\frac{-2\epsilon_2}{\epsilon_3}& \frac{\epsilon_2}{\epsilon_3} &  0& \ldots & 0 &  \ldots &\ldots &\ldots & 0
\\
\frac{\epsilon_2}{\epsilon_3} & \frac{-2\epsilon_2}{\epsilon_3}   &\ddots & \ddots  & \vdots &   & & & \vdots
\\
0 & \ddots  & \ddots & \ddots& 0 &   & & & \vdots
\\
\vdots &  \ddots &  \ddots &\frac{-2\epsilon_2}{\epsilon_3} & \frac{\epsilon_2}{\epsilon_3} & \ddots  & & & \vdots
\\
0 & \ldots   &0 &  \frac{\epsilon_2}{\epsilon_3}& 1 &  \frac{\epsilon_1}{\epsilon_3} &0 & \ldots   &0 
\\
\vdots &    & &  \ddots& \frac{\epsilon_1}{\epsilon_3} &  \frac{-2\epsilon_1}{\epsilon_3} &\ddots & \ddots   &\vdots
\\
\vdots &    & &  & 0 & \ddots &\ddots & \ddots   &0 
\\
\vdots &    & &  & \vdots & \ddots &\ddots & \frac{-2\epsilon_1}{\epsilon_3}   &\frac{\epsilon_1}{\epsilon_3} 
\\
0 & \ldots   & \ldots & \ldots & 0 & \ldots &0 & \frac{\epsilon_1}{\epsilon_3}   &\vphantom{\biggl(}\frac{-2\epsilon_1}{\epsilon_3} 
\end{pmatrix}.
\]

The Gramian matrix corresponding to $\mathfrak{S}_{1}$ is equal to the Cartan matrix of $\mathfrak{sl}_n$ multiplied by $-\frac{\epsilon_3}{\epsilon_2}$. Let us recall what commutes with the corresponding screening operators. First, note that operators $a_{r,j}$, $n+1\leq j \leq n+m$, clearly commutes with these screening operators, the nontrivial part of centralizer is written in terms of other generators. Namely, it was proven in \cite{FeiginFrenkel:1992} that $W$-algebra $\mathrm{W}(\dot{\mathfrak{sl}}_n)$ is a centralizer of such screening operators in the Heisenberg vertex algebra with generators $a_{r,i}-a_{r,i+1}$, $1\leq i <n$, the parameter $-\frac{\epsilon_3}{\epsilon_2}$ is responsible for central charge. The additional Heisenberg operators $\sum a_{r,i}$ also commutes with these screening operators, therefore the centralizer in the Heisenberg vertex algebra with generators $a_{r,i}$, $1\leq i\leq n$ coincides with $\mathrm{W}(\dot{\mathfrak{gl}}_n)$ \cite{Fateev Lukyanov}. 

Similarly, the centralizer of screening operators from  $\mathfrak{S}_{2}$ the Heisenberg vertex algebra with generators $a_{r,j}$, $n+1\leq j\leq n+m$ coincides with $\mathrm{W}(\dot{\mathfrak{gl}}_m)$.

The Gramian matrix corresponding to $\mathfrak{S}_{3}$ has blocks corresponding to $\mathfrak{sl}_n$, $\mathfrak{sl}_m$ and one additional column and row between them. This additional column and row corresponds to additional vector in $\mathfrak{S}_{3}$ which corresponds to  fermionic screening operator. Therefore we  call the $W$-algebra commuting with this system 
$\mathrm{W}(\dot{\mathfrak{gl}}_{n|m})$. The $\mathrm{W}(\dot{\mathfrak{gl}}_{n|1})$ case was considered in \cite{Feigin Semikhatov}, but we did not find any reference for general $n,m$. Note that our $W$-algebras differ from the ones introduced in~\cite{KRW}.

We expect that resolutions conjectured in Section \ref{ssec:Fnm} have conformal limit. We discuss their quantum group meaning in the next section.

\subsection{\!\!} \label{ssec:glnm}
A standard statement (conjecture) in the theory of vertex algebras is an equivalence of the abelian categories of certain representations of a vertex algebra and certain representations of a quantum group. This is a statement similar to the Drinfeld--Kohno or Kazhdan--Lusztig theorems  \cite{Kazhdan}. In fact, this equivalence of categories is an equivalence of braided tensor categories, but we do not need tensor structure here. 

It is widely believed that a certain category of representations of the vertex algebra $\mathrm{W}(\dot{\mathfrak{gl}}_{n})$ is equivalent to a certain category of representations of the quantum group $U_q\mathfrak{gl}_n\otimes U_{q'}\mathfrak{gl}_n$, where parameters $q,q'$ are given in terms of $\epsilon_1, \epsilon_2$ (people also use modular double of $U_q\mathfrak{gl}_n$). We conjecture that the same relation holds for the vertex algebra $\mathrm{W}(\dot{\mathfrak{gl}}_{n|m})$ and the quantum group $U_q\mathfrak{gl}_{n|m}\otimes U_{q'}\mathfrak{gl}_n\otimes U_{q''}\mathfrak{gl}_m$ for certain $q,q',q''$ given in terms of $\epsilon_1,\epsilon_2$.

Denote by $L^{(n)}_\nu$ the finite dimensional irreducible representation of $U_{q'}(\mathfrak{gl}_n)$, recall that these representations are labeled by partitions $\nu$ such that $l(\nu)\leq n$. Similarly denote by $L^{(m)}_\mu$ the finite dimensional irreducible representation of $U_{q''}\mathfrak{gl}_m$. 

For $U_q\mathfrak{gl}_{n|m}$ we will consider two types of representations. It is known (see \cite{{Berele Regev}}, \cite{Sergeev}) that  irreducible $U_q\mathfrak{gl}_{n|m}$ submodules of the tensor powers of $\mathbb{C}^{n|m}$ are labeled by partitions $\lambda$ such that $\lambda_{n+1} <m+1$, such irreducible modules are called tensor representation. We denote an analogous representation of $U_q\mathfrak{gl}_{n|m}$ by $L^{(n|m)}_\lambda$. Let $\mathfrak{p}\subset \mathfrak{gl}_{n|m}$ be a parabolic subalgebra with a Levi subgroup $\mathfrak{gl}_n\oplus \mathfrak{gl}_{m}$. The algebra $\mathfrak{p}$ acts on a tensor product of finite dimensional representations of $\mathfrak{gl}_n$ and $\mathfrak{gl}_{m}$ and a corresponding induced $\mathfrak{gl}_{n|m}$ module is called Kac module. We denote an analogues representation for $U_q{\mathfrak{gl}}_{n|m}$ by  $V_{\tilde{\nu},\tilde{\mu}}$, where $\tilde{\nu},\tilde{\mu}$ label the finite dimensional representations of $\mathfrak{gl}_n$ and $\mathfrak{gl}_{m}$ correspondingly.

We conjecture that under the aforementioned equivalence the tensor product of irreducible modules $L^{(n|m)}_\lambda \otimes L^{(n)}_\nu\otimes L^{(m)}_\mu $ goes to the conformal limit of $\mathcal{N}^{n,m}_{\mu,\nu,\lambda}(v)$, and tensor product $V_{\tilde{\nu},\tilde{\mu}} \otimes L^{(n)}_\nu\otimes L^{(m)}_\mu$ goes to the conformal limit of 
$\mathcal{N}^{n,0}_{\varnothing,\nu,\tilde{\nu}}(v_1)\otimes \mathcal{N}^{0,m}_{\mu,\varnothing,\tilde{\mu}}(v_2)$.\footnote{Here we are a bit sloppy about parameter $v$ in $\mathcal{N}^{n,m}_{\mu,\nu,\lambda}(v)$, actually partition $\nu$ in $L^{(n)}_\nu$ can consist of non integer parts with integer differences (instead of $\nu$ in $\mathcal{N}^{n,m}_{\mu,\nu,\lambda}(v)$) and total shift of these parts depends on $v$.}

In paper \cite{Cheng Kwon Lam} Cheng, Kwon and Lam constructed a resolution in terms of the Kac modules of the tensor module of $\mathfrak{gl}_{n|m}$. Taking the (conjectural) $q$-deformation of this resolution we have a complex which consists of modules $V_{\tilde{\nu},\tilde{\mu}}$ with the cohomology $L^{(n|m)}_\lambda$. Multiplying by $L^{(n)}_\nu\otimes L^{(m)}_\mu$ we get a complex of $U_q\mathfrak{gl}_{n|m}\otimes U_{q'}\mathfrak{gl}_n\otimes U_{q''}\mathfrak{gl}_m$ modules. Then, applying the above equivalence we get a resolution of the conformal limit of $\mathcal{N}^{n,m}_{\mu,\nu,\lambda}(v)$ in terms of the conformal limits of $\mathcal{N}^{n,0}_{\varnothing,\nu,\tilde{\nu}}(v_1)\otimes \mathcal{N}^{0,m}_{\mu,\varnothing,\tilde{\mu}}(v_2)$. This resolution should be a materialization of \eqref{eq:chi:bos}, its $q$-deformation was discussed above in Section \ref{ssec:Fnm}.

The Euler characteristic of the resolution constructed in \cite{Cheng Kwon Lam} yields the following formula 
\[
s_{\lambda}(x|y)=\sum_{\alpha_1\geq \alpha_2 \geq \ldots \geq \alpha_r \geq r-m} 
(-1)^{\sum \alpha_i}s_{\pi+m-r,-\alpha}(x)s_{\alpha,\kappa}(y) \prod_{1 \leq i \leq n, 1\leq j \leq m}\left(1+\frac{y_j}{x_i}\right).
\]
Here the notation $\pi$, $\kappa$ were introduced in Section \ref{ssec:Notation}, $s_{\mu}(x)$ is a Schur polynomial, i.e., the character of $L^{(m)}_\mu$ and $s_{\lambda}(x|y)$ is a hook Schur polynomial (or super-Schur polynomial), i.e., the character of $L^{(n|m)}_\lambda$. This formula resembles our character formula~\eqref{eq:chi:bos}.

\begin{Remark}Moens and van der Jeugt in the paper \cite{MvdJ} found another formula for the character of $L_\lambda^{(n|m)}$
\begin{equation}\label{eq:MvdJ}
s_{\lambda}(x|y)= \frac{(-1)^{mn-r}\,\prod\limits_{i,j}\left(1+\dfrac{y_j}{x_i}\right)}{V(x_1,\ldots,x_n) V(y_1,\ldots,y_m)} 
\det\begin{pmatrix}
\left(\sum\limits_{a\geq 0} (-1)^a x_j^{-a-1+m}y_i^{a}\right)_{\substack{1 \leq i \leq m\\ 1\leq j \leq n}} 
& \left(y_i^{Q_j}\right)_{\substack{1 \leq i \leq m\\ 1 \leq j \leq m-r}} 
\\ 
\left(x_j^{P_i+m}\right)_{\substack{1 \leq i \leq n-r\\ 1 \leq j \leq n}} & 0
\end{pmatrix}.
\end{equation}
This formula is similar to our formula \eqref{eq:chi:LGV}. It is natural to conjecture that there is a resolution which is a materialization of \eqref{eq:MvdJ} and under the equivalence this resolution goes to a resolution which is a materialization of \eqref{eq:chi:LGV}.
\end{Remark}

\end{document}